\documentclass[reqno]{amsart}

\usepackage{amssymb,amsthm,amsmath,verbatim}
\usepackage{t1enc}
\usepackage{enumerate}

\usepackage{tikz}

\numberwithin{equation}{section}

\newtheorem{prop}{Proposition}[section]
\newtheorem{dfn}[prop]{Definition}
\newtheorem{lm}[prop]{Lemma}
\newtheorem{fact}[prop]{Fact}
\newtheorem{cor}[prop]{Corollary}
\newtheorem{thm}[prop]{Theorem}
\newtheorem{clm}[prop]{Claim}

\newtheorem{obs}[prop]{Observation}
\newtheorem{subclaim}{Subclaim}[prop]
\newtheorem{prob}[prop]{Problem}

\newcommand{\mc}[1]{\mathcal{#1}}
\newcommand{\mb}[1]{\mathbb{#1}}
\newcommand{\bs}{\mathbb{B}}
\newcommand{\oo}{\omega}
\newcommand{\uhr}{\upharpoonright}

\def\<{\left\langle}
\def\>{\right\rangle}
\def\br#1;#2;{\bigl[ {#1} \bigr]^ {#2} }
\newcommand{\ooone}{{{\omega}_1}}
\newcommand{\ooto}{{\ooone\times {\omega}}}

\newcommand{\mbb}[1]{\mathbb{#1}}

\newcommand{\pref}[1]{(P\ref{#1})}

\newcommand{\setm}{\setminus}
\newcommand{\empt}{\emptyset}
\newcommand{\subs}{\subset}
\newcommand{\dom}{\operatorname{dom}}
\newcommand{\ran}{\operatorname{ran}}
\newcommand{\supp}{\operatorname{supp}}

\newcommand{\rhop}{{\underline{\rho}}}

\title{Partitioning bases of topological spaces}
\author[D.T. Soukup]
{D\'aniel T. Soukup}

\address{University of Toronto}
\email{daniel.soukup@mail.utoronto.ca}
\urladdr{http://math.utoronto.ca/$\sim  $dsoukup}

\author[L. Soukup]
{Lajos Soukup*}\thanks
  {The preparation of this paper was partially
supported by  OTKA grant K 83726.} 
 \thanks{*Corresponding author}
\address
      {Alfr{\'e}d R{\'e}nyi Institute of Mathematics, Hungarian Academy of
Sciences, Budapest, Hungary  }
\email{soukup@renyi.hu}
\urladdr{http://www.renyi.hu/$\sim  $soukup}
\subjclass[2010]{54A35, 03E35, 54A25}
\keywords{base, resolvable, partition }

\begin{document}
\maketitle

 \begin{abstract}
We investigate whether an arbitrary base for a dense-in-itself topological space
can be partitioned into two bases. We prove that every base for a $T_3$
Lindel\"of topology can be partitioned into two bases while there exists a
consistent example of a first-countable, 0-dimensional, Hausdorff space  of size
$2^\oo$ and weight $\omega_1$ which admits a point countable base without a
partition to two bases. 
\end{abstract} 

\section{Introduction}

At the Trends in Set Theory conference in Warsaw, Barnab\'as Farkas\footnotemark
\footnotetext{personal communication} raised the natural question whether one
can partition any given base for a topological space into two bases; we will
call this property being \emph{base resolvable}. Note that every space with an
isolated is not base resolvable; hence, from now on by \emph{space} we mean a
\emph{dense-in-itself topological space}.  The aim of this paper is to present
two streams of results: in the first part of the article, we will show that
certain natural classes of spaces are base resolvable. In the second part, we
present a method to construct non base resolvable spaces.

The paper is structured as follows: in Section \ref{gen}, we will start with
general observations about bases and we prove that metric spaces and weakly separated spaces are base resolvable. This section also serves as an
introduction to the methods that will be applied in Section \ref{Lindsec} where
we prove one of our main results in Theorem \ref{Lindres}: every $T_3$ (locally)
Lindel\"of space is base resolvable.

In Section \ref{comb}, we investigate base resolvability from a purely
combinatorial viewpoint which leads to further results: every hereditarily
Lindel\"of space (without any separation axioms) is base resolvable and any base
for a $T_1$ topology which is closed under finite unions can be partitioned into
two bases, see Theorem \ref{herLind} and \ref{finunion} respectively.

Next in Theorem \ref{negl}, we prove that every base $\bs$ for a space $X$
(resolvable or not) contains a large \emph{negligible} portion, i.e. there is
$\mc U\in [\bs]^{|\bs|}$ such that $\bs\setminus \mc U$ is still a base for $X$.

The second part of the paper starts with Section \ref{irressec}; here, we
isolate a partition property, denoted by $\mb P \to (I_\omega)^1_2$, of the
partial order $\mb P=(\bs, \supseteq)$ associated to a base $\bs$ which is
closely related to base resolvability. We will construct a partial order $\mb P$
with this property in Theorem \ref{part} and deduce the existence of a $T_0$ non
base resolvable topology (in ZFC) in Corollary \ref{T0irres}.

Next, in Section \ref{conirres} we present a ccc forcing (of size $\oo_1$) which
introduces a first-countable, 0-dimensional, Hausdorff space $X$ of size $2^\oo$
and weight $\omega_1$ such that $X$ is not base resolvable. The main ideas of
the construction already appear in Section \ref{irressec} however the details
here are much more subtle and the proofs are more technical.

The paper finishes with a list of open problems in Section \ref{problems}. We
remark that Section \ref{conirres} was prepared by the second author and the
rest of the paper is the work of the first author.

The first author would like to thank his PhD advisor, William Weiss, the long
hours of useful discussions. Both authors are grateful for the help of all the
people they discussed the problems at hand, especially Allan Dow, Istv\'an 
Juh\'asz, Arnie Miller, Assaf Rinot, Santi Spadaro, Zolt\'an Szentmikl\'ossy and
Zolt\'an Vidny\'anszky. Finally, we thank Barnab\'as Farkas for the excellent
question!

\section{General results} \label{gen}

In this section, we prove some basic results concerning
{partitions of families of sets and }
 partitions of bases;
these proofs will introduce us to the more involved techniques of the upcoming
sections.

\begin{dfn}
{We say that a family of sets $\mc A$  is \textbf{well-founded} iff the poset
$\<\mc A,\supset\>$ is well-founded, i.e. there is no strictly decreasing infinite chain 
$A_0\supsetneq A_1\supsetneq A_2\supsetneq \dots $   in $\mc A$.}

{
$\mc A $  is \textbf{weakly increasing} iff there is a well order
$\prec$ of $\mc A$ such that $A\prec B$ implies that $B\setminus
A\neq\emptyset$. 
}
\end{dfn}

\begin{prop}\label{pr:well_founded}
{Every family of sets $\mc A$ contains a weakly increasing, and so well-founded subfamily 
$\mc B$ with 
\begin{equation}\notag
 \bigcup \mc A=\bigcup \mc B. 
\end{equation}
 } 
\end{prop}

\begin{proof}
Fix an  arbitrary well-ordering $\prec$  of $\mc A$ and let
 \begin{equation}\label{eq:weakly}
  \mc B=\{B\in \mc A: B\setm A\ne \empt \text{ for all $A\prec B$} \}.
 \end{equation}
If $C\prec B$ for $C,B\in \mc B$, then $B\setm C\ne \empt$, so $\prec$ witnesses
that $\mc B$ is weakly increasing.

To  verify $\bigcup \mc A=\bigcup \mc B$ pick an arbitrary $p\in \bigcup \mc A$
and let 
\begin{equation}
 B=\min_{\prec}\{A\in \mc A: p\in A\}.
\end{equation}
Then $p\in B\setm A$ for all $A\prec B$, so $B\in \mc B.$ 
Thus $\bigcup \mc A=\bigcup \mc B$.
\end{proof}

\begin{dfn}A base $\mb B$ for a space $X$ is \textbf{resolvable} iff it can be
decomposed into two bases. A space $X$ is \textbf{base resolvable} if every base
of $X$ is resolvable.
\end{dfn}

Recall that by \emph{space} we will mean a dense-in-itself topological space
throughout the paper.

Partitioning sets with additional structure is a highly investigated theme in
mathematics; let us cite a classical result  of A. H. Stone which is relevant to
our case:

 \begin{thm}[A. H. Stone, \cite{stone}] \label{stone} Every partially ordered
set $(\mb P,\leq)$ without maximal elements can be partitioned into two cofinal
subsets.
\end{thm}

\begin{prop}\label{first} 
Suppose that $(X,\tau)$ is a topological space and $p\in X$.
\begin{enumerate}[(1)]
\item Every neighborhood base at  $p$ can be partitioned into two neighborhood bases.
\item Every $\pi$-base can be partitioned into two $\pi$-bases.
\item If $\mc B$ is a neighborhood base at  $p$ and
$\mc B=\mc B_0\cup \mc B_1$ then either $\mc B_0$ or
$\mc B_1$ is a neighborhood base at $p$. 
\item If $\mc B$ is a base and $\mc U\subs \mc B $ is well founded then 
$\mc B\setm \mc U$ is a base.
\item Every base can be partitioned into a cover and a base.
\end{enumerate}
\end{prop}

\begin{proof}
(1) and (2) follow from Theorem \ref{stone}.

Indeed, 
write $\tau_x=\{U\in \tau:x\in U\}$ for $x\in X$ 
and observe that  $\mc B\subs \tau_x$ is a  
neighborhood base at  $x$ iff $\mc B$ is cofinal in $\<\tau_x,\supset\>$.
By Theorem \ref{stone}, every neighborhood base  at $p $ can be partitioned
into two cofinal subsets of $\<\tau_p,\supset\>$, i.e. into two    
neighborhood bases  at $p$.
So (1) holds.

To prove  (2),  observe that $\mc B\subs \tau$ is a  
$\pi$-base  iff $\mc U$ is cofinal in $\<\tau,\supset\>$.
By Theorem \ref{stone}, every $\pi$-base can be partitioned
 into two cofinal subsets, i.e. into two  $\pi$-bases.

\noindent  (3)
If $\mc B_0$ is not a neighborhood base at $p$ then there is 
  an    element  $V\in \tau_p$ which does not contain any element of
$\mc B$. Thus $\mc B\cap \mc P(V)=\mc B_1\cap \mc P(V)$,
so  $\mc B_1$ is a neighborhood base at $p$.

\noindent  (4)
Let  $x\in X$. 
Then $\tau_x \cap \mc B$ is  a  neighborhood base at $x$.
Since  $\tau_x\cap \mc U$ is well-founded, 
$\tau_x\cap \mc U$ is not a  neighborhood base at $x$.
Thus, by (3), $\tau_x \cap (\mc B\setm \mc U)$ is  a  neighborhood base at $x$.

Since $x$ was arbitrary, we proved that $\mc B\setm \mc U$ is a base.

\noindent  (5)
Every base $\mc B$ contains a well-founded cover
$\mc U$ by Proposition \ref{pr:well_founded} while 
$\mc B\setm \mc U$ is still a base of
$X$ by (4).
\end{proof}

A family $\mb B$
of open subset of a space $\<X,\tau\>$ is a base iff every nonempty open set
is the union of some subfamily of $\mb B$.
This fact implies the 
following:

\begin{obs}\label{fillobs}
 Suppose that $(X,\tau)$ is a topological space, $\mb B_i\subs \tau$ for $i<2$ and $\mb B_0$ is a base.
\begin{enumerate}
 \item If for every $U\in \mb B_0$ there is $\mc U\subs \mb B_1$ with $U=\cup \mc U$ then $\mb B_1$ is a base as well.
\item If $X$ is $T_3$ and for every $U,V\in \mb B_0$ with  $\bar U\subs V$ there is $\mc U\subs \mb B_1$ with $\bar U\subs \cup \mc U\subs V$ then $\mb B_1$ is a base as well.
\end{enumerate}

\end{obs}

Now we prove our first general result.

\begin{prop} Every space with a $\sigma$-disjoint base is base resolvable; in
particular, every metrizable space is base resolvable.
\end{prop}
\begin{proof} Fix a space $X$ with a base $\cup\mb \{\mb E_n:n\in\omega\}$ where $\mb E_n$ is
a disjoint family of open sets for each $n\in\omega$; fix an arbitrary base $\bs $ as well which we aim to
partition.

By induction on $n\in\omega$, construct $\bs _{i,n}\subseteq \bs$ for $i<2$ such
that 
\begin{enumerate}[(1)]
\item $\bs_{i,n}$ is well founded for $i<2$, $n\in\oo$,
\item $\bs_{i,n}\cap \bs_{j,m}=\emptyset$ if $i,j<2$, $n,m\in\omega$ and
$(i,n)\neq(j,m)$,
\item for every $V\in \mb E_n$ and $i<2$ there is $\mc U\subseteq \bs_{i,n}$
such that $\cup\mc U= V$. 
\end{enumerate}

Assume that $\{\bs_{i,k}:i<2,k<n\}$ was constructed.
By Proposition \ref{first}(4)  property (1) assures that $\bs\setm
\cup\{\bs_{i,k}:i<2,k<n\}$ is
still a base of $X$. 
Thus,  by 
Proposition \ref{pr:well_founded}, for each $E\in \mb E_n$ we can choose
a well-founded family $\mc U_E\subs \bs\setm
\cup\{\bs_{i,k}:i<2,k<n\}$ such that $E=\bigcup \mc U_E.$
Let $$\bs_{0,n}=\bigcup\{\mc U_E: E\in \mb E_n\}.
$$ 
Since the elements of $\mb E_n$ are pairwise disjoint, 
$\bs_{0,n}$ is well-founded as well.

To obtain 
$\bs_{1,n}$ repeat the construction of $\bs_{0,n}$ 
using  $\bs\setm
(\cup\{\bs_{i,k}:i<2,k<n\}\cup \bs_{0,n})$ instead of $\bs\setm
\cup\{\bs_{i,k}:i<2,k<n\}$.

 Let
$\bs_i=\cup \{\bs _{i,n}:n\in\oo\}$ for $i<2$; property (3) and Observation
\ref{fillobs}(1) implies that $\mb B_i$ is a base for $i<2$.
\end{proof}

Note that every $\sigma$-disjoint base is point countable, on the other hand our example
of an irresolvable base constructed in Section  \ref{conirres} is point
countable.

A somewhat similar technique, which will be used later as well, gives the
following result:

\begin{prop}\label{weakLindres} Suppose that a regular space $X$ satisfies
$L(X)<\kappa=w(X)=\min\{\chi(x,X):x\in X\}$. Then $X$ is base resolvable.
\end{prop}

Recall that $L(X)$, the \emph{Lindel\"of number of $X$}, is the minimal cardinality $\kappa$ such that every open cover of $X$ contains a subcover of size $\kappa$. The \emph{weight of $X$}  is $$w(X)=\min \{|\mb B|:\mb B \text{ is a base of }X\}$$ and the \emph{character of a point $x\in X$} is $$\chi(x,X)=\min \{|\mc U|:\mc U \text{ is a neighbourhood base of }x\}.$$

\begin{proof}It is well known that any base contains a base of size $w(X)$; therefore it suffices to show that any base $\mb B$ of size $w(X)$ can be partitioned into two bases. Let us fix an enumeration
$\{(U_\alpha,V_\alpha):\alpha<\kappa\}$ of all pairs of elements $U,V\in \mb B$
such that $\overline{U}\subseteq V$.

By induction on $\alpha<\kappa$ construct pairwise disjoint
families $$\{\mb B_{0,\alpha},\mb
B_{1,\alpha}:{\alpha}<{\kappa}\}\subseteq  \br \mb B;\le L(X);$$ such that
\begin{equation}\label{eq:smallL}
\text{$\overline{U_\alpha}\subseteq \cup \mb B_{i,\alpha} \subseteq V_\alpha$ for every $i<2$.}
 \end{equation}

Since 
the cardinality of the family 
$\mb B_{<\alpha}=\bigcup\{\mb B_{i,\beta}:\beta<\alpha,i<2\}$
is at most $L(X)\cdot |{\alpha}|$ 
and $L(X)\cdot |\alpha|<\min\{\chi(x,X):x\in X\}$,
the family  $\mb B_{<\alpha}$
can not contain a neighborhood base at any point $x\in X$. 

Thus, by Proposition \ref{first}, 
 $\mb B\setminus \mb B_{<\alpha}$
is still a base for $X$ for every $\alpha<\kappa$. 
It follows that the induction
can be carried out as we can select disjoint 
$\mb B_{{\alpha},0}$  and $\mb B_{{\alpha},0}$ from 
$ \br \mb B\setm \mb B_{<{\alpha}};\le L(X);$ 
so that $$\overline{U_\alpha}\subseteq \cup \mb B_{{\alpha},i} \subseteq V_\alpha$$ 
for $i<2$.

Thus the disjoint families $\bs_i=\cup\{\mb
B_{i,\alpha}:\alpha<\kappa\}$ form a base for $X$ by property (\ref{eq:smallL}) above and Observation \ref{fillobs}(2); thus $X$ is base resolvable.
\end{proof}

We end this section by a simple observation; recall that a space $X$ is \emph{weakly separated} if there is a neighborhood assignment $\{U_x:x\in X\}$ (meaning that $U_x$ is a neighbourhood of $x$) so that $x\neq y\in X$ implies that $x\notin U_y$ or $y\notin U_x$. Note that left-or right separated spaces are weakly separated as well as the Sorgenfrey line.

\begin{obs} Every weakly separated space is base resolvable.
\end{obs}
\begin{proof} Recall that every neighborhood base at some point $x$ can be partitioned into two
neighbourhood bases by Proposition \ref{first}(1). Thus, if $\mb B$ is a base of $X$
and there is a disjoint family $\{\bs_x:x\in X\}$ of subsets of $\bs$ such that $\bs_x$ is a neighbourhood base at $x$ for any $x\in X$ then by partitioning $\bs_x$ for each $x\in X$ into two
neighbourhood bases of $x$ we get a partition of $\bs $ into two bases of $X$.

Now, let us fix a base $\bs$ we wish to partition and a neighbourhood assignment $\{U_x:x\in X\}$ witnessing that $X$ is weakly separated. Define $$\bs_x=\{U\in \bs:x\in U\subs U_x\}$$ for $x\in X$; clearly, $\bs_x$ is neighbourhood base at $x$. Furthermore, if $x\neq y$ and say $x\notin U_y$ then $U\in \bs_x$ implies $U\notin \bs_y$; that is, $\bs_x\cap \bs_y=\emptyset$ if $x\neq y\in X$ which finishes the proof.
\end{proof}

We thank the referee for pointing out this last observation for us.

\section{Lindel\"of spaces are base resolvable} \label{Lindsec}

Our aim in this section is to prove that $T_3$ Lindel\"of spaces are base
resolvable; we start with a definition and some observations while the most
important part of the work is done in the proof of Lemma \ref{ext}. 

\begin{dfn}  Let $\mc{A},\mc{B}$ families of open sets in a space $X$. We say
that $\mc{A}$ \textbf{weakly fills} $\mc{B}$ iff for every $U,V\in\mc{B}$ such
that $\overline U\subs V$ there is $\mc W\subseteq \mc A$ such that
$$\overline{U}\subseteq \cup\mc W\subs V.$$ $\mc{A},\mc{B}$ is called a
\textbf{weakly good pair} iff $\mc{A},\mc{B}$ are disjoint, $\mc{A}$ weakly
fills $\mc{B}$ and $\mc{B}$ weakly fills $\mc{A}$.
\end{dfn}

We remark that in the next section we introduce stronger notions called
\emph{filling} and \emph{good pairs}. 
The first part of the following observations basically restates 
Observation \ref{fillobs}(2) with our new terminology:

\begin{obs}\label{wgpobs} Suppose that $X$ is a regular space.
\begin{enumerate}[(1)]
\item If $(\mc{A},\mc{B})$ is a weakly good pair in $X$ then $\mc A$ contains a
neighborhood base at $x$ iff $\mc B$ contains a neighborhood base at x, for any
$x\in X$.
\item If $\{\mc A_\alpha:\alpha<\kappa\}$ and $\{\mc B_\alpha:\alpha<\kappa\}$
are increasing chains and $(\mc{A} _\alpha,\mc{B}_\alpha)$ is a weakly good pair in $X$
then $(\cup_{\alpha<\kappa} \mc A_\alpha,\cup_{\alpha<\kappa} \mc B_\alpha)$ is
a weakly good pair as well.
\end{enumerate}
\end{obs}

We say that the weakly good pair $(\mc A', \mc B')$ \textbf{extends} the weakly
good pair $(\mc A, \mc B)$ iff $\mc A \subseteq \mc A'$ and $\mc B\subs \mc B'$.
A family of 
pairs  $\{(\mc A_\xi, \mc B_\xi):\xi<\Theta\}$ is
\textbf{pairwise disjoint} iff $\mc A_\xi\cap \mc B_\zeta=\emptyset$ for each
$\xi,\zeta<\Theta$.

Next, we prove that weakly good pairs can be nicely extended in Lindel\"of
spaces.

\begin{lm}\label{ext} Suppose that $X$ is a $T_3$ Lindel\"of space with a base
$\mb B$. Given a weakly good pair $(\mc{A},\mc{B})$ from elements of 
$\mb B$ and a
single pair of open sets $\{U,V\}$ such that $\overline U\subs V$ there is a
weakly good pair  $(\mc{A}',\mc{B}')$ formed by elements of $ \mb B$ extending $(\mc{A},\mc{B})$
such that both $\mc{A}'$ and $\mc{B}'$ weakly fills $\{U,V\}$.
\end{lm}
\begin{proof} We will show this essentially by induction on the size of 
$\mc{A}$ and $\mc{B}$ however we need to prove something significantly stronger
(and more technical) than the statement of the lemma itself. 

 Let $\triangle_\kappa$ stand for the following statement: \underline{for each}
pairwise disjoint family of weakly good pairs  $\{(\mc A_i, \mc B_i), (\mc
C_j,\mc D_j):i<n,j<k\}$, each a subfamily from $\mb B$, such that $|\mc
A_i|,|\mc B_i|\leq \kappa$ and  arbitrary  family of open sets $\mc E$ of size at most
$\kappa$ \underline{there is} a  weakly good pair $(\mc A, \mc B)$ from $\mb B$
of size at most $\kappa $ such that
\begin{enumerate}
\item $\cup_{i<n} \mc A_i\subs \mc A$ and $\cup_{i<n} \mc B_i\subs \mc B$,
\item $\mc A$ and $\mc B$   weakly fill $\mc E$,
\item $\{(\mc A,\mc B), (\mc C_j,\mc D_j):j<k\}$ is still pairwise disjoint.
\end{enumerate}
We prove that $\triangle_\kappa$ holds for every infinite $\kappa$ by induction
on $\kappa$.
\begin{clm}$\triangle_\omega$ holds.
\end{clm}

\begin{proof}Fix  $\{(\mc A_i, \mc B_i), (\mc C_j,\mc D_j):i<n,j<k\}$ and  $\mc
E$ as above. By induction on $m\in\omega$ we build increasing chains $\{\mc A^m:m\in
\oo\}$ and $\{\mc B^m:m\in\oo\}$ from subsets of $\bs$ such that 
\begin{enumerate}[(1)]
\item $\mc A^0=\cup_{i<n} \mc A_i$, $\mc B^0=\cup_{i<n}\mc B_i$,
\item $\mc A^{m+1}\setm \mc A^m$ and $\mc B^{m+1}\setm \mc B^m$ are countable
well-founded families, 
\item
the family of pairs $\{(\mc A^m,\mc B^m), (\mc C_j,\mc D_j):j<k\}$ is  pairwise disjoint
\end{enumerate} 
for each $m\in\oo$.
Furthermore, we want to make sure that $\mc A=\cup_{m\in\oo} \mc A^m$ and $\mc
B=\cup_{m\in\oo} \mc B^m$ form a weakly good pair and they both weakly fill
$\mc E$. Therefore, we partition $\oo$ into infinite sets $\omega=\cup
\{D_m:m\in\omega\}$ and at the $m^{\rm th}$  step 
\begin{enumerate}[(1)]\addtocounter{enumi}{3}
 \item we fix a surjective map $$f_m:D_m\setm
(m+1)\to \{(U,V)\in  (\mc A^m\cup \mc B^m\cup \mc E)^2: \overline U\subs V \};$$
\item 
if $m\in D_\ell\setm (\ell+1)$ and $f_\ell(m)=(U,V)$ then   
both $\mc A^{m+1}$ and $\mc B^{m+1}$  weakly fill $\{U,V\}$.
 \end{enumerate}

In particular, it suffices to construct disjoint $\mc A^{m+1}$ and $\mc B^{m+1}$ from $\mc
A^{m}$ and $\mc B^{m}$ such that they satisfy (2), (3) and (5) above, especially  
they both weakly fill a given $(U,V)$. We construct $\mc A^{m+1}$, the proof for $\mc
B^{m+1}$ is analogous.

\begin{subclaim}$\bs \setm (\mc B^m \cup \bigcup_{j<k} \mc D_j)$
is a base of $X$.
\end{subclaim}

\begin{proof}[Proof of the Subclaim]

Let $x\in X$ be arbitrary.

If $\mc B^m \cup \bigcup_{j<k} \mc D_j$
does not contain a neighborhood base at $x$, then
$\bs \setm (\mc B^m \cup \bigcup_{j<k} \mc D_j)$   
should   contain a neighborhood base at $x$ by Proposition \ref{first}(3).

Assume know that $\mc B^m \cup \bigcup_{j<k} \mc D_j$
contains a neighborhood base at $x$. Since 
\begin{equation}\notag
\mc B^m \cup \bigcup_{j<k} \mc D_j= (\mc B^m\setm \mc B^0)
\cup\bigcup_{i<n}
\mc B_i\cup \bigcup_{j<k}\mc D_j,
\end{equation}
applying Proposition \ref{first}(3) again, one of the sets
\begin{equation}
\mc B^m\setm \mc B^0,\mc B_0,\dots, \mc B_{n-1},\mc D_0,\dots, \mc D_{k-1}
\end{equation}
contains a neighborhood base at $x$.
Since
$\mc B^m\setm \mc B^0 $ is well-founded,
it can not contain a neighborhood base.
If $\mc B_i$ (or $\mc D_j$, respectively)  contains a neighborhood base at $x$, then 
$\mc A_i$ (or $\mc C_j$, respectively)
also contains a neighborhood base at $x$ by Observation \ref{wgpobs}(1).
In both cases,   $\bs \setm (\mc B^m \cup \bigcup_{j<k} \mc D_j)$
contains a neighborhood base, which proves the Subclaim.
\end{proof}

Since $X$ is Lindel\"of, using the Subclaim above and Proposition
\ref{pr:well_founded}  we can find  a countable well-founded cover 
$\mc Q\subs \bs \setm (\mc B^m \cup \bigcup_{j<k} \mc D_j)$  of $\overline U$
with $\cup \mc Q\subs V$. Now define  
 $\mc A^{m+1}=\mc A^m\cup \mc Q$. Since $\mc Q$
and $(\mc B^m \cup \bigcup_{j<k} \mc D_j)$ are disjoint, (3) holds. 
(2) and (5) are clear from the construction.
\end{proof}

\begin{clm}Suppose that $\triangle_\lambda$ holds for every
$\omega\leq\lambda<\kappa$. Then $\triangle_\kappa$ holds.
\end{clm}
\begin{proof} Fix  $\{(\mc A_i, \mc B_i), (\mc C_j,\mc D_j):i<n,j<k\}$ and  $\mc
E$, let $\operatorname{cf}(\kappa)=\mu$ and fix a cofinal sequence of ordinals
$(\kappa_\xi)_{\xi<\mu}$ in $\kappa$. Take a chain of elementary submodels
$(M_\xi)_{\xi<\mu}$ of $H(\theta)$ (where $\theta$ is large enough) such that everything relevant is in $M_0$, $\kappa_\xi\subs
M_\xi$ and $|M_\xi|=|\kappa_\xi|$ for $\xi<\mu$. The following is an easy
consequence of $M_\xi$ being elementary and $X$ being Lindel\"of:

\begin{subclaim}  $(\mc A_i\cap M_\xi, \mc B_i\cap M_\xi)$ is a weakly good pair
and $|\mc A_i\cap M_\xi|,|\mc B_i\cap M_\xi|\leq|\kappa_\xi|$ for all $i<n$.
\end{subclaim}

\begin{proof}[Proof of the Subclaim]
If $U,V\in \mc A_i\cap M_\xi$, $\overline U\subs V$  
then $\mc A_i,\mc B_i,A\in M_\xi $ implies that
\begin{equation}\notag
 M_\xi \vDash \exists \mc B\in \br \mc B_i;{\omega};\quad 
\overline U\subs  \bigcup \mc B\subs V.
\end{equation}
because $X$ is Lindel\"of. So there is $\mc B\in  M_\xi\cap\br \mc B_i;{\omega};$
such that $\overline U\subs  \bigcup \mc B\subs U$. Since $\mc B$ is countable,
$\mc B\in  M_\xi$ implies $\mc B\subs  M_{\xi}$. 
So we have  $\mc B\subs \mc B_i\cap M_{\xi}$ with $\overline U\subs  \bigcup \mc B\subs V
$. This shows that $\mc B_i\cap M_{\xi}$ fills $\mc A_i\cap M_\xi$ and the other direction of the proof is completely analogous.
\end{proof}

By induction on $\xi<\mu$ construct weakly good pairs
$\{(\mc A^\xi, \mc B^\xi): \xi<\mu\}$ so that $\mc A^\xi\subs \mc A^\zeta$, $\mc B^\xi\subs \mc B^\zeta$ for $\xi<\zeta<\mu$ and
\begin{enumerate}[(i)]
\item $\cup_{i<n} (\mc A_i\cap M_\xi)\subs \mc A^\xi\subs \mb B$ and $\cup_{i<n}
(\mc B_i\cap M_\xi)\subs \mc B^\xi\subs \mb B$,
\item $\mc A^\xi$ and $\mc B^\xi$ has size $\leq |\kappa_\xi|$,
\item $\mc A^\xi$ and $\mc B^\xi$ weakly fills $\mc E\cap M_\xi$,
\item $\mc A^\xi\cap \mc B_i=\emptyset, \mc A^\xi\cap \mc D_j=\emptyset$ and
$\mc B^\xi\cap \mc A_i=\emptyset, \mc B^\xi\cap \mc C_j=\emptyset$.
\end{enumerate}
This can be done using $\triangle_{|\kappa_\xi|}$ at stage $\xi$. First note
that $\mc A^{<\xi}=\cup \{\mc A^\zeta:\zeta<\xi\}$ and $\mc B^{<\xi}=\cup \{\mc
B^\zeta:\zeta<\xi\}$ are of size at most $|\kappa_\xi|$ and  $(\mc A^{<\xi},\mc
B^{<\xi})$ is a weakly good pair. Also, the family $$\{(\mc A^{<\xi},\mc
B^{<\xi}),(\mc A_i\cap M_\xi, \mc B_i\cap M_\xi);(\mc A_i, \mc B_i), (\mc
C_j,\mc D_j):i<n,j<k\}$$ is pairwise disjoint. Hence $\triangle_{|\kappa_\xi|}$
implies that there is a weakly good pair $(\mc A^\xi, \mc B^\xi)$ from $\mb B$
of size at most $|\kappa_\xi|$ which fills $\mc E\cap M_\xi$ and is pairwise
disjoint from $\{(\mc A_i,\mc B_i),(\mc C_j,\mc D_j):i<n,j<k\}$ while $$\mc
A^{<\xi}\cup \bigcup_{i<n} (\mc A_i\cap M_\xi)\subs \mc A^\xi$$
and
$$\mc B^{<\xi}\cup \bigcup_{i<n} (\mc B_i\cap M_\xi)\subs \mc B^\xi.$$
Note that $\triangle_{|\kappa_\xi|}$ was used to find the common extension of
$n+1$ weakly good pairs such that this extension is disjoint from $n+k$ given
weakly good pairs.
 Now define $\mc A=\cup\{\mc A^\xi:\xi<\zeta\}$ and $\mc B=\cup\{\mc
B^\xi:\xi<\zeta\}$; $(\mc A,\mc B)$ is the desired extension.
\end{proof}
This finishes the proof the lemma.

\end{proof}

Recall that a space is  \emph{locally Lindel\"of} if every point has a neighbourhood with Lindel\"of closure.

\begin{prop}
 Suppose that $X$ is a $T_3$ locally Lindel\"of space. Then $X$ embeds into a $T_3$ Lindel\"of space $X^*$ with $|X^*\setm X|= 1$.
\end{prop}
\begin{proof}
 Construct $X^*$ on the set $X\cup\{x^*\}$ where neighborhoods of the point
$x^*$ are of the form $\{x^*\}\cup X\setm \overline U$ with $U\subs X$ open such that
there is an open $V\subs X$ with $\overline U\subs V$ and $\overline V$ is Lindel\"of. It
is clear that $X^*$ is Hausdorff and Lindel\"of.

Note that if $U,V$ are open in $X$, $\overline U\subs V$ and 
$\overline V$ is Lindel\"of, then  $\overline V$ is normal as well, so  there is an open
$W\subs V$ so that $\overline U \subs W\subs \overline W \subs V$. 
So $X^*$ is regular at the point $x^*$, so  
 $X^*$ is regular.
\end{proof}

\begin{cor} \label{Lindres} Every $T_3$ locally Lindel\"of space is base
resolvable. In particular, every $T_3$ locally countable or locally compact
space is base resolvable. 
\end{cor}
\begin{proof} Fix a base $\bs$ for a $T_3$ Lindel\"of space $X$  and consider
the set $\mb P$ of all weakly good pairs $(\mc A,\mc B)$ from $\bs$ partially
ordered by extension. Note that we can apply Zorn's lemma to $\mb P$ by
Observation \ref{wgpobs} part (2); pick a maximal weakly good pair $(\mc A, \mc B)\in \mb P$. Lemma
\ref{ext} implies that a maximal weakly good pair must weakly fill every pair of open sets $\{U,V\}$ with
$\overline U \subs V$, hence both $\mc A$ and $\mc B$ are bases of $X$. 

Given a $T_3$ locally Lindel\"of space $X$ with a base $\bs$ consider it's
one-point Lindel\"ofization $X^*=X\cup \{x^*\}$ with the base $$\bs ^*=\bs \cup
\{U\subseteq X^*:U \text{ is an open neighbourhood of  } x^* \text{ in } X^* \}.$$ $X^*$ is $T_3$
Lindel\"of hence base resolvable; thus $\bs^*$ can be partitioned into two bases,
$\mb B^*_0$ and $\mb B^*_1$,  
which clearly gives a partition of $\bs$: 
$\mb B^*_0\cap \mb B$ and $\mb B^*_1\cap \mb B$.
\end{proof}

\section{Combinatorics of resolvability} \label{comb}

In this section, we will prove a combinatorial lemma which will be our next tool
in showing that further classes of space are base resolvable.

\begin{dfn} Let $\mc{A},\mc{B}\subseteq \mc P(X)$. We say that $\mc{A}$
\textbf{fills} $\mc{B}$ iff $$U=\cup\{V\in \mc{A}:V\subsetneq U\}$$ for every
$U\in\mc{B}$. $\mc{A},\mc{B}$ is called a \textbf{good pair} iff $\mc{A},\mc{B}$
are disjoint, $\mc{A}$ fills $\mc{B}$ and $\mc{B}$ fills $\mc{A}$. $\mc A$ is
\textbf{self-filling} if $\mc A$ fills $\mc A$.
\end{dfn}

Note that 
if $\mc A\subseteq \mc P(X)$
fills $\{\cap \mc B:\mc B\in[\mc A]^{<\omega}\}$ and $\mc A$ covers $X$
then  $\mc A$ is a base for a topology 
on $X$. 

\begin{dfn}A self-filling family $\mc A$ is \textbf{resolvable} iff there is a
partition $\mc A_0,\mc A_1$ of $\mathcal A$ such that $\mc A_i$ fills $\mc A$ for $i<2$.
\end{dfn}

The importance of the following lemma is that it shows that resolvability is a \emph{local} property:

\begin{thm}\label{gp} Suppose that $\mb{B}\subseteq \mc P(X)$ 
is self-filling. Then the
following are equivalent:
 \begin{enumerate}[(1)]
	\item for every $U\in \mb{B}$ there is a good pair
$(\mb{B}_0^U,\bs_1^U)$ from $ \bs$ such that $$U=\cup\bs^U_0=\cup\bs^U_1,$$
	\item $\bs$ is resolvable.
\end{enumerate}
\end{thm}

\begin{proof} (2) implies (1) is trivial.

To see  that (1) implies (2),  let $\mc P$ be the set of all good pairs $(\mb B_0, \mb B_1)$ formed by elements
of $\mb B$; $\mc P$ is partially ordered by $(\mb B_0, \mb B_1)\leq(\mb B'_0,
\mb B'_1)$ iff $\mb B_i \subseteq \mb B'_i$ for $i<2$. It is clear that every
chain in $(\mc P, \leq)$ has an upper bound hence, by Zorn's lemma, we can pick
a $\leq$-maximal element $(\mb B_0, \mb B_1)\in \mc P$. 

We claim that $\mb B_i$ fills $\mb B$ for $i<2$. Pick any $U\in \mb B$ and
consider the good pair  $\mb{B}_0^U,\bs_1^U$ with $U=\cup\bs^U_0=\cup\bs^U_1$.
Define $$\mb B'_i=\mb B_i \cup \mb (\mb B^U_i\setminus \mb B_{1-i})$$ for $i<2$.

The second statement of the following lemma yields immediately that  
 $(\mb B'_0, \mb B'_1)$ forms a good pair which fills
$\{U\}$.

\begin{lm}\label{lm:extend_fill}
(1)If a family of sets $\mc A$ fills a family  of sets  $\mc B$ and $\mc A'$ fills 
$\mc B'$ then   $\mc A\cup (\mc A'\setm \mc B)$ fills $\mc B'$.

\noindent (2)
If $(\mc A,\mc B)$ and $(\mc A',\mc B')$ are good pairs then  
$(\mc A\cup (\mc A'\setm \mc B),\mc B\cup (\mc B'\setm \mc A))$ is also a good pair
which fills $\cup\mc B'$.
\end{lm}

\begin{proof}[Proof of the Lemma]
(1) Pick  $U\in \mc B'$.
Since $\mc A'$ fills 
$\mc B'$, there is 
$\mc A^+\subs \mc A'\setm \{U\}$ with $U=\cup \mc A^+$.
For each $B\in \mc A^+\cap \mc B$ choose $\mc A_B\subs \mc A$
with $B=\cup\mc A_B$. Finally let 
\begin{equation}\notag
\mc A^*=(\mc A^+\setm \mc B)\cup \bigcup\{\mc A_B:B\in \mc A^+\cap \mc B\}. 
\end{equation}
Then $\mc A^*\subs \mc A\cup (\mc A'\setm \mc B)\setm\{U\}$ and 
\begin{multline}\notag
 \cup A^*=\bigcup \left((\mc A^+\setm \mc B)\cup \bigcup\{\mc A_B:B\in \mc A^+\cap \mc B\}
\right)=\\
\bigcup \left((\mc A^+\setm \mc B)\cup \{B:B\in \mc A^+\cap \mc B\}
\right)=\bigcup\mc A^+=U.
\end{multline}

\noindent (2)
The families $\mc A\cup (\mc A'\setm \mc B)$ and $\mc B\cup (\mc B'\setm A)$
are clearly disjoint, $\mc A\cup (\mc A'\setm \mc B)$ fills 
$\mc B\cup (\mc B'\setm A)\cup\{\bigcup \mc A\}$ and 
$\mc B\cup (\mc B'\setm A)$ fills $\mc A\cup (\mc A'\setm \mc B)\cup\{\bigcup \mc B\}$ by (1)
which was to be proved.
\end{proof}

Also, $(\mb B_0, \mb B_1)\leq(\mb B'_0, \mb B'_1)$ thus by the
maximality of  $(\mb B_0, \mb B_1)$ we have that $\mb B'_i=\mb B_i$. This
finishes the proof.
\end{proof}

The first corollary is a direct application and shows that resolvability is
preserved by unions.

\begin{cor}Suppose that $\mb B_\alpha$ is a resolvable self-filling family for
each $\alpha<\kappa$. Then $\cup\{\mb B_\alpha:\alpha<\kappa\}$ is a resolvable
self-filling family as well.
\end{cor}

\begin{cor} Suppose that a self-filling family $\bs$ has the property that 
\begin{enumerate}
 \item[$(\dag)$] for
every $U\in\bs$ there is $\mc U\in [\mb B\setm\{U\}]^{\leq\oo}$ such that
$U=\cup \hspace{0.03cm }\mc U$.
\end{enumerate}
  Then $\bs$ is resolvable.
\end{cor}

\begin{proof}
We need the following Claim.

\begin{clm}
If $\mc A\subs \bs$ is well-founded  then for every $W\in \bs$
there is a countable  well-founded family $\mc B(W,\mc A)\subs  \bs\setm \mc A$ with 
$\cup \mc B(W,\mc A)=W$.
\end{clm}

\begin{proof}
We can assume that $W\in \mc A$.
By ($\dag$) there is a  countable self-filling family $\mc C\subs  \bs $ 
 with  $W\in \mc C$.  Let
\begin{equation}\notag
   \mc V=\{V\in \mc C\setm \mc A: V\subsetneq W\}.
\end{equation}
Since $\mc A$ is well-founded, for each  $x\in W$ the family 
 $\{Z\in \mc A\cap \mc C:x\in Z\}$
 has a  $\subset$-minimal element $Z$.
Since $\mc C$ is self-filling, there is $V\in \mc C$ with $x\in V\subsetneq Z$.
Then $V\in \mc V$.

Thus $\bigcup \mc V=W$. Now, by Proposition \ref{pr:well_founded},
there is a well-founded family $\mc B(W,\mc A)\subs \mc V$ with 
$\bigcup \mc V=\bigcup \mc B(W,\mc A)$.
\end{proof}

By Theorem \ref{gp}, it suffices to prove that for every $U\in \bs$ 
there is a  good pair $(\mb{B}_0,\bs_1)$ from $ \bs$ such that 
$U=\cup\bs_0=\cup\bs_1$.

Fix a $U\in \bs$. Partition $\oo$ into infinite sets $\omega=\cup
\{D_m:m\in\omega\}$. 
By induction on $m\in\omega$ we build increasing chains $\{\mb B^m_0:m\in
\oo\}$ and $\{\mb B^m_1:m\in\oo\}$ from subsets of $\bs$ such that
\begin{enumerate}[(1)]
  \item $\mb B^0_0=\mb B^0_1=\empt$,
\item $\mb B^m_0$ and $\mb B^m_1$ are disjoint, well founded and countable families,
\item fix a surjective map 
$$f_m:D_m\setm (m+1)\twoheadrightarrow \{U\}\cup B^m_0\cup B^m_1,$$
\item if $m\in D_\ell$ and $f_\ell(m)=V$ then 
\begin{equation}
 \mb B^{m+1}_0=\mb B^{m}_0\cup  \mc B(V,\mb B^{m}_1)
\end{equation}
and 
\begin{equation}
 \mb B^{m+1}_1=\mb B^{m}_1\cup  \mc B(V,\mb B^{m+1}_0).
\end{equation}
\end{enumerate}

Let  $\bs_i=\cup\{\bs^m_i:m\in\oo\}$ for $i<2$.
The  $(\mb{B}_0,\bs_1)$  is a  good pair and $U=\cup\bs_0=\cup\bs_1$. 
Indeed, if $V\in \mb{B}_i\cup\{U\}$ then
$V\in \mb{B}^m_i\cup\{U\}$ for some $m\in {\omega}$ and so 
$f_m(\ell)=V$ for some $\ell\in D_m\setm (m+1)$. 
Thus there is a family $\mc B\subs \mb B^{\ell+1}_{1-i}\subs \mb B_{1-i}$ 
with $\bigcup \mc U=V$.
\end{proof}

\begin{cor}\label{herLind} Locally countable or hereditarily Lindel\"of spaces
are base resolvable without assuming any separation axioms.
\end{cor}

Our next corollary establishes that every reasonable space admits a resolvable
base.

\begin{cor} \label{finunion} Suppose that $\bs$ is a base closed under finite
unions in a $T_1$ topological space. Then $\bs$ is resolvable.
\end{cor}
\begin{proof}We apply Theorem \ref{gp} again: fix $U\in \bs$ and we construct a
good pair covering $U$. Fix an arbitrary  strictly decreasing sequence
$\{U_n:n\in\omega\}\subseteq \bs$  such that $U_ 0\subseteq U$ and fix $y_n\in U_{n-1}\setm U_n$ for $n\in \omega\setm \{0\}$. Let
$$\bs^U_i=\{V\in\bs\cap \mc P(U): \exists k\in\omega\setm \{0\}:U_{2k+i}\subseteq V \text{
but } U_{2k-1+i}\not\subseteq V\}$$ for $i<2$. It should be clear that $\bs^U_0\cap \bs^U_1=\emptyset$.

Next we prove that $U=\cup \bs ^U _i$ for $i<2$. Fix $i<2$ and note that $\{U_{2k+i}:k\in\omega\setm \{0\}\}\subs \bs ^U _i$. Now fix $x\in U$ and we prove that $x\in \cup \bs ^U _i$; without loss of generality we can suppose that $x\notin U_{2+i}$. Find any $k\in \omega$ so that $y_{2k+i}\neq x$ and take $W\in \bs$ so that $x\in W\subs U\setm \{y_{2k+i}\}$; here we used that $\bs$ is a base of a $T_1$ topology. Note that $V=U_{2k+i}\cup W\in \bs$ as $\bs$ is closed under finite unions and that $x\in V\in \bs ^U_i$.

Finally we show that $(\bs^U_0,
\bs^U_1)$ is a good pair; we will show that $\bs ^U_0$ fills $\bs ^U_1$, the other direction is completely analogous. Fix $V\in \bs ^U_1$ and fix a point $z\in V$. Find an $l\in \omega$ so that $U_{2l-1}\subs V$ and $z\neq y_{2l}$. As $\bs$ is a base, there is $W\in\bs$ so that $z\in W\subs V\setm \{y_{2l}\} $. Let $V'=U_{2l}\cup W$; as $\bs$ is closed under finite unions we have $V'\in \bs$, moreover $V'\in \bs ^U_0$ as witnessed by $U_{2l}\subs V'$ but $U_{2l-1}\not\subseteq V'$. Finally, $z\in V'\subs V$ as we wanted. 
\end{proof}

\begin{cor}\label{T1top} The set of all open sets in a $T_1$ topological space
is resolvable.
\end{cor}

Let MA(Cohen) denote Martin's axiom
restricted to the partial orderings of the form 
$Fn({\kappa}, 2, {\omega})$ for some ${\kappa}$ 
where, $Fn({\kappa}, 2,{\omega})$ is the poset   of functions from some
finite subset of ${\kappa}$ to 2 ordered by reverse inclusion.

\begin{cor}\label{martin}Under MA(Cohen) every space $X$ of local size
$<2^\omega$ is base resolvable without assuming any separation axioms.
\end{cor}
\begin{proof}Fix a base $\bs$ of $X$; we may assume that $|U|<2^\omega$ for all $U\in \bs$. We apply Theorem \ref{gp} to prove that $\bs$ is resolvable as a self filling family which in turn will imply that $\bs$ is a resolvable base. Fix $U\in \bs$ and we construct a good
pair covering $U$. Let $\kappa=|U|$ and select $\bs_U\in [\bs]^{\kappa}$ which fills itself and $\cup
\bs_U=U$. Now consider the ccc partial order $\mb P= Fn(\bs_U,2, \omega)$, i.e.
the set of all finite partial functions from $\bs_U$ to 2. Now consider
$$D_{x,V,i}= \{f\in \mb P: \text{ there is } W\in f^{-1}(i): x\in W \subset
V\}$$ for $i<2$ and $ x\in V\in \mb B_U$; note that each $D_{x,V,i}$ is dense in
$\mb P$. Hence there is a filter $G\subseteq \mb P$ which intersects 
$D_{x,V,i}$ for $i<2$ and $x\in V\in \mb B_U$. Let 
$\mb B_i=\{V\in \mb B_U: (\cup
G)(V)=i\}$ for $i<2$ and note that $(\mb B_0,\mb B_1)$ is the desired good pair.
\end{proof}

\section{Thinning self filling families} \label{thinning}

Let $\mbb B$ be a self filling family; note that $\bs$ is \emph {redundant} in
the sense that $\bs \setminus \mc U$  still fills $\bs$ for a finite or more
generally, a well founded family $\mc U$. 
\begin{dfn} We say that $\mc U\subseteq \bs$ is \emph{negligible} iff $\bs
\setminus \mc U$  still fills $\bs$.
\end{dfn}

Our aim in this section is to show that every self filling family $\bs$ contains
a negligible subfamily of size $|\bs|$. Note that a base $\bs$ for a space $X$
is resolvable iff it contains a negligible subfamily $\mc U\subseteq \bs$ such
that $\mc U$ is a base of $X$ as well. We will make use of the following
definitions:

\begin{dfn}
 If $\bs$ fills itself then let $$L(U,\bs)=\min\{|\mc V|:\mathcal V\subseteq \bs\setm
\{U\},U=\cup \mc V\}$$ for $U\in \bs$.
\end{dfn}

\begin{obs} \label{neglobs}Suppose that $\bs$ fills itself and $\mc U \subseteq
\bs$. 
 \begin{enumerate}[(1)]
 \item\label{fill_negligible} 
If $\bs \setminus \mc U$ fills $\mc U$ then $\mc U$ is negligible.
  \item\label{well_founded_negligible} 
If $\mc U$ is well founded then $\bs \setminus \mc U$ fills $\mc U$ and so $\mc U$ 
is negligible; in particular,  if $\mc U$ is weakly increasing, then 
$\mc U$ is negligible.
 \end{enumerate}

\end{obs}

Our first proposition establishes the main result for self filling families $\bs$ with $cf |\bs|=|\bs|$.

\begin{prop}
 Suppose that $\bs$ fills itself and $\kappa=|\bs|$ is regular. Then $\bs$
contains a negligible family of size $\kappa$.
\end{prop}
\begin{proof} We can suppose that $L(U,\bs)<\kappa$ for every $U\in\bs$;
otherwise we can find a weakly increasing subfamily of size $\kappa$ which is
negligible by (1) and (2) of Observation \ref{neglobs}.

 It suffices to define a sequence $\mc U_\xi,\mc V_\xi\in [\bs]^{<\kappa}$
for $\xi<\kappa$ such that 
\begin{enumerate}[(1)]
\item $\mc U_\xi\cap\mc V_\xi=\emptyset$,
\item $\mc U_\xi\subs \mc U_\zeta$ and $\mc V_\xi\subs \mc V_\zeta$ for $\xi<\zeta<\kappa$,
 \item $\mc V_\xi$ fills $\mc U_\xi$, and 
\item  $\mc U_{\xi+1}\setm \mc U_\xi\neq\empt$;
\end{enumerate}
 Clearly, $\mc U=\cup\mc \{\mc U_\xi: \xi<\kappa\}$ will be a
negligible set of size $\kappa$ in $\bs$ by (3) of Observation \ref{neglobs}.
Suppose we have $\mc U_\xi,\mc V_\xi\in [\bs]^{<\kappa}$ for $\xi<\zeta$ as
above for some $\zeta<\kappa$; then $\bs \setminus \cup\{\mc U_\xi,\mc
V_\xi:\xi<\zeta\}\neq\emptyset$ by $\kappa$ being regular hence we can select
$U_\zeta\in \bs \setminus \cup\{\mc U_\xi,\mc V_\xi:\xi<\zeta\}$ and define
$$\mc U_\zeta=\bigcup \{\mc U_\xi:\xi<\zeta\}\cup\{U_\zeta\}.$$ Find $\mc W
\subseteq \bs\setm \{U_\zeta\}$ of size $<\kappa$ such $\cup \mc W =U_\zeta$;
define $$\mc V_\zeta =\bigcup \{\mc V_\xi:\xi<\zeta\}\cup (\mc W\setminus \mc U_\zeta).$$ 

Since $\bigcup \{\mc V_\xi:\xi<\zeta\}$ fills
$\bigcup \{\mc U_\xi:\xi<\zeta\}$ by the inductive hypothesis (3) above,
Lemma \ref{lm:extend_fill}(1) implies that   
$\mc V_\zeta$ fills $\mc U_\zeta$.
\end{proof}

\begin{thm}\label{negl}
  Suppose that $\bs$ fills itself. Then $\bs$ contains a negligible family of
size $|\bs|$.
\end{thm}
\begin{proof}
 We can suppose that $\mu=\operatorname{cf}({\kappa})<\kappa=|\bs|$ and that every weakly increasing
sequence in $\bs$ is of size less than $\kappa$
by Observation \ref{neglobs}(\ref{well_founded_negligible}).
Fix a cofinal strictly
increasing sequence of regular cardinals $(\kappa_\xi)_{\xi<\mu}$ in $\kappa$ such that $\mu
< \kappa_0$ and define $$\bs_\xi=\{U\in\bs:L(U,\bs)\leq \kappa_\xi\}$$
 for every $\xi<\mu$. So 
\begin{equation}\label{eq:B=} 
\bs=\bigcup_{\xi<\mu}\bs_\xi. 
\end{equation}
If there is a $\xi<\mu$ such that every weakly increasing sequence in $\bs$ is of size less than
$\kappa_\xi$ then $\bs=\bs_\xi$. Let us define a set mapping $F:\bs \to
[\bs]^{<\kappa_\xi^+}$ such that $U=\cup F(U)$ where $F(U)\subseteq \bs\setm
\{U\}$. As $\kappa_\xi^+<\kappa$ we can apply Hajnal's Set Mapping theorem (see
Theorem 19.2 in \cite{hajnal}): there is an $F$-free set $\mc U$ of size
$\kappa$ in $\bs$, i.e. $F(U)\cap \mc U=\emptyset$ for all $U\in \mc U$; observe
that $\mc U$ is negligible as $\cup \{F(U):U\in\mc U\}\subseteq \bs \setm \mc U$
fills $\mc U$.

Now we suppose that $\bs\neq \bs_\xi$ for $\xi<\mu$, that is there is a weakly increasing sequence in $\bs$ of size $\kappa_\xi$ for all $\xi<\mu$. It suffices to define sequences $\mc U_\xi,\mc
V_\xi\in [\bs]^{<\kappa}$ for $\xi<\mu$ such that 
\begin{enumerate}[(i)]
\item $\mc U_\xi\subs \mc U_\zeta$ and $\mc V_\xi\subs \mc V_\zeta$ for $\xi<\zeta<\kappa$,
\item $\mc U_\xi,\mc V_\xi$ are disjoint and $\kappa_\xi\leq |\mc U_\xi|$,
\item $\mc V_\xi$ fills $\mc U_\xi$.
\end{enumerate}
Indeed, the union $\cup\{\mc U_\xi:\xi<\mu\}$ is negligible in $\bs$ of size
$\kappa$  by Observation \ref{neglobs}\ref{fill_negligible}. 
because $\cup\{\mc V_\xi:\xi<\mu\}$ fills $\cup\{\mc U_\xi:\xi<\mu\}$.

Suppose we defined $\mc U_\xi,\mc V_\xi\in [\bs]^{<\kappa}$ for
$\xi<\zeta$; let $$\lambda=\bigl(|\bigcup \{ \mc U_\xi\cup \mc
V_\xi:\xi<\zeta\}|\cdot \kappa_\zeta\bigr)^+.$$ Note that $\lambda<\kappa$ thus
we can pick a weakly increasing family $\mc W\in [\bs]^\lambda$; without loss of
generality, we can suppose that $\mc W$ is disjoint from $\bigcup \{ \mc
U_\xi\cup \mc V_\xi:\xi<\zeta\}$. Note that $$\mc W=\cup\{\bs_\delta\cap \mc W:
\delta<\mu\}$$ 
by (\ref{eq:B=}),
and that $\mu<\operatorname{cf}(\lambda)=\lambda$, hence there is $\delta<\mu$
such that $\mc W'=\mc W\cap \bs_\delta$ has size $\lambda$. Define 
$$\mc U_\zeta=\mc
\bigcup\{\mc U_\xi:\xi<\zeta\}\cup \mc W'.$$

Now, for every $U\in \mc W'$ select $F(U)\in [\bs\setm \{U\}]^{\le \kappa_\delta}$
such that $U=\cup F(U)$. Define $$\mc V_\zeta=\bigcup\{\mc
V_\xi:\xi<\zeta\}\cup\bigcup\{F(U):U\in \mc W'\}\setminus \mc U_\zeta.$$ 
Note
that $\kappa_\zeta\leq |\mc U_\zeta|=\lambda$  and $|\mc V_\zeta|\leq
\lambda\cdot \kappa_\delta<\kappa$. It is only left to prove that $\mc V_\zeta$
fills $\mc U_\zeta$; in fact, it suffices to show that $\mc V_\zeta$ fills $\mc
W'$. Suppose that $\prec$ is the well ordering witnessing that $\mc W'$ is
weakly increasing and suppose that there is a $U\in \mc W'$ which is not filled
by $\mc V_\xi$; we can suppose that $U$ is $\prec$-minimal. Fix an $x\in U$
witnessing that $\mc V_\zeta$ does not fill $U$. Pick $V\in F(U)$ such that
$x\in V\subs U$; 
then $V\notin \mc V_\zeta$, so
$V\in \mc W'$ or  $V\in  \bigcup \{\mc U_\xi:\xi<\zeta\}$;
if $V\in \mc W'$ then $V\prec U$, thus $V$ is filled by $\mc
V_\zeta$ by the minimality of $U$. This contradicts the choice of $x$, hence
$V\notin \mc W'$. Thus $V\in  \bigcup \{\mc U_\xi:\xi<\zeta\}$
which is filled by $\bigcup \{\mc V_\xi:\xi<\zeta\}\subs \mc V_\zeta$ by 
the inductional hypothesis; this again contradicts the choice of $x$, which
finishes the proof.

\end{proof}

\section{Irresolvable self filling families} \label{irressec}

The aim of this section is to construct an irresolvable self filling family and
deduce the existence of a non base resolvable $T_0$ topological space.

Given a partial order $(\mb P, \leq)$ and $p,q\in \mb P$ let $$[p,q]=\{r\in \mb
P:p\leq r\leq q\}.$$ The key to our construction is the following special partition relation:

\begin{dfn}
 We say that a poset $\mb P$ without maximal elements satisfies $$\mb P \to
(I_\omega)^1_2$$ iff for every partition $\mb P=D_0\cup D_1$ there is $i<2$ and
strictly increasing $\{p_n:n\in\omega\}\subseteq D_i$ such that
$[p_0,p_n]\subseteq D_i$ for every $n\in \omega$. The negation is denoted by
$\mb P \nrightarrow (I_\omega)^1_2$.
\end{dfn}

The above definition is motivated by the following:

\begin{obs}
 For any irresolvable self filling family $\mb B\subseteq \mc P (X)$ the partial
order $\mb P=(\mb B, \supseteq)$ satisfies  $\mb P \to (I_\omega)^1_2$.
\end{obs}
\begin{proof}Consider a partition of $\mb P=(\mb B, \supseteq)$ into sets $D_0,
D_1$; as $\bs$ is irresolvable, there is $i<2$, $x\in X$ and $U\in D_i$ such
that $V\in D_i$ for every $V\in\bs$ with $x\in V\subseteq U$. Pick a strictly
decreasing sequence $\{V_n:n\in\oo\}\subseteq \bs$ such that $x\in V_n\subseteq
U$ for every $n\in\oo$; clearly, $[V_0,V_n]\subseteq D_i$ for every $n\in\oo$.
\end{proof}

Our next aim is to find a partial order $\mb P$ first with  $\mb P \to
(I_\omega)^1_2$; note that trees or $Fn(\kappa,2)$ cannot satisfy $\mb P \to
(I_\omega)^1_2$. Moreover:

\begin{prop}
  $\mb P \nrightarrow (I_\omega)^1_2$ for every countable poset $\mb P$ without
maximal elements.
\end{prop}
\begin{proof} Fix a countable poset $\mb P$ without maximal elements. We construct a partition $\mb P=P_0\cup P_1$ witnessing $\mb P \nrightarrow (I_\omega)^1_2$ as follows: first, fix an enumeration $\{I_n:n\in\omega\}$ of all intervals $I=[p',p]$ in $\mb P$ which contain an infinite chain
and let $\mb P= \{p_n:n\in\omega\}$ denote a 1-1 enumeration. Construct disjoint $P_{0,n},P_{1,n}\subseteq \mb P$ by induction on $n\in \omega$ such that
\begin{enumerate}[(i)]
 \item $P_{i,n}$ is a finite union of antichains for $i<2$,
\item $p_n\in P_{0,n}\cup P_{1,n}$,
\item $I_n\cap P_{i,n}\neq\emptyset$ for $i<2$,
\item whenever $C=\{c_k:k\in\omega\}\subseteq \mb P$ is a strictly increasing chain, 
$p_n\in C$ and 
$[c_i,c_j]$ is well-founded (i.e.   $[c_i,c_j]\notin I$) for all $i<j<{\omega}$
then
$$\bigcup_{k\in\omega}[c_0,c_k] \cap P_{i,n}\neq \emptyset$$ for each $i<2$.
\end{enumerate}
Provided we can carry out this induction, we have that 
\begin{clm}
 $\mb P \nrightarrow (I_\omega)^1_2$.
\end{clm}
\begin{proof}
 Let $P_i=\cup \{P_{i,n}:n\in\oo\}$ for $i<2$ and note that this is a partition
of $\mb P$ by (ii). Consider an arbitrary strictly increasing chain
$C=\{c_k:k\in\omega\}\subseteq \mb P$. If there is $k\in \omega$ such that
$[c_0,c_k]$ contains an infinite chain in $\mb P$ then there is an $n\in \omega$
such that $I_n=[c_0,c_k]$; property (iii) from above ensures that $P_i\cap
[c_0,c_k]\neq \emptyset$ for $i<2$. Otherwise, 
the intervals $[c_i,c_j]$ are all well-founded
intervals; in this case, property (iv) ensures that
$\bigcup_{k\in\omega}[c_0,c_k] \cap P_{i}\neq \emptyset$ for $i<2$.
\end{proof}

Now suppose we constructed $P_{i,n-1}$ satisfying the above conditions for $i<2$; note that
finitely many elements can be added to both $P_{0,n-1}$ and $P_{1,n-1}$ without
violating (i), thus (ii) and (iii) are easy to satisfy (note that $I_n\setminus(P_{0,n-1}\cup P_{1,n-1})$ is infinite since $I_n$ contains an infinite chain).

It suffices to show the following to finish our proof:

\begin{clm}
 Fix $p\in \mb P$ and $A\subseteq \mb P$ which is covered by finitely many
antichains. Then there is an antichain $B\subseteq \mb P\setminus A$ such that
whenever $C=\{c_k:k\in\omega\}\subseteq \mb P$ is a strictly increasing chain, $p\in C$ and 
the intervals $[c_i,c_j]$ are all well-founded
 then $$\bigcup_{k\in\omega}[c_0,c_k]
\cap B\neq \emptyset.$$
\end{clm}

\begin{proof}
Let 
\begin{equation}\notag
R=\{q\in P:p\le q\text{ and $[p,q]$  does not contain infinite chains}\}. 
\end{equation}
Then $\<R,\le\>$ is well founded, so we can define, by well-founded recursion, 
a rank function  $rk$ from $R$ into the ordinals such that
\begin{equation}\label{eq:rk}
\begin{array}{ll}
rk(p)=0,\\  rk(t)=\sup\{rk(s)+1:s\in [p,t)\}  & \text{ if  $t\in R$, $p<t$. }
\end{array}
\end{equation}

 Let $Q=R\setm A$ and define $q^-$
to be the element minimizing $rk$ on $[p,q]\setminus A$ for $q\in Q$.
Let $$B=\{q^-:q\in Q\}.$$ 

First note that $B$ is an antichain by (\ref{eq:rk}). Now fix a strictly
increasing chain $C=\{c_k:k\in\omega\}\subseteq P$ 
such that the intervals $[c_i,c_j]$ are all well-founded
and  $p\in C$; 
since $A$  is covered by finitely many
antichains  there is $q\in C\setminus A$ such that $p<q$;
also, $q\in Q$ by $[p, q]$ being well founded. Thus $q^-\in
\bigcup_{k\in\omega}[c_0,c_k] \cap B$.
\end{proof}

Indeed, to finish the inductive construction, apply the claim twice 
to find antichain $B_0\subseteq \mb P\setminus A$ and 
$B_1\subseteq \mb P\setminus (A\cup B_0)$ such that
$\bigcup_{k\in\omega}[c_0,c_k]
\cap B_i\neq \emptyset$ whenever $C=\{c_k:k\in\omega\}\subseteq \mb P$ is a 
strictly increasing chain,  $p\in C$ and 
the intervals $[c_i,c_j]$ are all well-founded.

Then  $P_{0,n}=P_{0,n-1}\cup B_0$ and 
$P_{1,n}=P_{1,n-1}\cup B_1$ are appropriate extensions satisfying (iv).
\end{proof}

 We will call a countable strictly increasing sequence of elements of a poset $\mb P$ a
\emph{branch}; we say that a branch $x=(x_n)_{n\in\oo}$ goes above an element
$p\in \mb P$ iff $p\leq x_n$ for some $n\in\omega$.

\begin{thm}\label{part}
 There is a partial order $\mb P$ of size $\omega_1$ without maximal elements
such that $\mb P \to (I_\omega)^1_2$. Furthermore, 
\begin{enumerate}[(1)]
\item every $p\in\mb P$ has finitely many predecessors,
\item if $p\nleq q$ in $\mb P$ then there is a branch $x$ in $\mb P$ which goes
above $q$ but not $p$.
\end{enumerate}
\end{thm}
\begin{proof}
Let us fix a function $c:[\oo_1]^2\to \omega$ such that $c(\cdot,\zeta):\zeta\to
\omega$ is 1-1 for every $\zeta\in\oo_1$. It is easy to see that such functions
satisfy the following:

\begin{fact}\label{unbound}If  $c(\cdot,\zeta):\zeta\to \omega$ is 1-1 for every
$\zeta\in\oo_1$ for some $c:[\oo_1]^2\to \omega$ then for every uncountable,
disjoint family $\mc A \subseteq [\oo_1]^{<\oo}$ and $N\in \oo$ there are
$a<b$\ \footnote[1]{$a<b$ iff $\xi<\zeta$ for all $\xi\in a,\zeta\in b$} in $\mc
A$ such that $c(\xi,\zeta)>N$ for every $\xi\in a,\zeta\in b$.
\end{fact}

 Also, fix an enumeration $\{(y_\alpha,w_\alpha):\oo\leq\alpha<\oo_1\}$ of all
pairs of elements of $\oo_1\times \oo$ such that $y_\alpha, w_\alpha\in
\alpha\times \omega$. 

We define $\mb P=(\oo_1\times \oo, \leq)$ as follows: by induction on $\alpha\in
L_1$ (where $L_1$ stands for the limit ordinals in $\oo_1$) we construct a poset
$\mb P_\alpha=((\alpha+\oo)\times \oo,\leq_\alpha)$ with properties:
\begin{enumerate}[(i)]
\item $\mb P_\alpha$ has no maximal elements and every $p\in \mb P_\alpha$ has
finitely many predecessors,
\item $\leq_\alpha\upharpoonright\beta=\leq_\beta$ for all $\beta<\alpha$,
 \item $(\xi,n)<_\alpha(\zeta,m)$ implies that $\xi<\zeta$ and
$\max(n,c(\xi,\zeta))<m$,
\item there is a $t_\alpha\in\mb P_\alpha$ such that $t<_\alpha t_\alpha$ 
if and only if $t\leq_\alpha y_\alpha$ or $t\leq_\alpha w_\alpha$ for any $t\in \mb P_\alpha$,
\item if $p\nleq q$ in $\mb P_\alpha$ then there is a branch $x$ in $\mb
P_\alpha$ which goes above $q$ but not $p$.
\end{enumerate}
We only sketch the inductive step: suppose that $y_\alpha=(\xi, n)$ and
$w_\alpha=(\zeta, m)$.  Let $\Gamma=\{\nu<\omega_1:$ there is $s\leq y_\alpha$ or $s\leq
w_\alpha$ with $s=(\nu, l)$ for some $l\in\oo\}$ and note that $|\Gamma|<\omega$ by (i). Let $$k=\max\{n,m,c(\nu,\alpha):\nu\in\Gamma\}+1.$$ Now
define $t_\alpha=(\alpha,k)$ and $\leq_\alpha$ so that  $t<_\alpha t_\alpha$
implies that $t\leq_\alpha y_\alpha$ or $t\leq_\alpha w_\alpha$. Extend
$\leq_\alpha$ further so that $\mb P_\alpha$ has no maximal elements and
satisfies (v); this can be done by "placing" copies of $2^{<\oo}$ above elements
of $\mb P_\alpha\setm \cup\{\mb P_\beta:\beta<\alpha\}$.

Let us define $\mb P=\cup \{\mb P_\alpha:\alpha<\oo_1\}$ and $\leq=\cup
\{\leq_\alpha:\alpha<\oo_1\}$; observe that $(\mb P,\leq)$ is well defined and
trivially satisfies (1) and (2). In what follows, $\pi_{\oo_1}$ and $\pi_{\oo}$
denotes the projections from $\omega_1\times \oo$ to the first and second
coordinates respectively.
\begin{clm}
 $\mb P \to (I_\omega)^1_2$.
\end{clm}
\begin{proof}
Suppose that $\mb P=D_0\cup D_1$; we can assume that $D_0$ and $D_1$ are both
cofinal in $\mb P$. Now suppose that there is no increasing chain with each interval in one
of the $D_i$ and reach a contradiction as follows. We will say that an interval
$[s,t]$ in $\mb P$ is \emph{i-maximal} for some $i<2$ if $[s,t]\subseteq D_i$
but $[s,t']\nsubseteq D_i$ for every $t<t'\in \mb P$. Observe that for every $s\in D_i$
there is $t\in D_i$ such that $[s,t]$ is i-maximal; otherwise, we can construct
an increasing chain starting from $s$ with each interval in $D_i$. Now construct
increasing 4-element sequences $R_\alpha=\{\tilde x_\alpha, \tilde y_\alpha
, \tilde z_\alpha, \tilde w_\alpha\}\subseteq \mb P$ for $\alpha<\oo_1$
such that $\tilde x_\alpha\leq \tilde y_\alpha
\leq \tilde z_\alpha \leq \tilde w_\alpha$ and
\begin{enumerate}[(a)]
 \item  $[\tilde x_\alpha,\tilde y_\alpha]\subseteq \mb P_0$ is a 0-maximal
interval,
\item $[\tilde z_\alpha,\tilde w_\alpha]\subseteq \mb P_1$ is a 1-maximal
interval,
\item $\pi_{\oo_1}''R_\alpha < \pi_{\oo_1}'' R_\beta$ if $\alpha<\beta$.
\end{enumerate}
By passing to a subsequence of $\{R_\alpha:\alpha<\oo_1\}$ we can suppose that
the image of $(\tilde x_\alpha, \tilde y_\alpha
, \tilde z_\alpha, \tilde w_\alpha)$ under $\pi_\omega$ is independent of 
$\alpha<\omega_1$ and we let $N=\max \pi_\omega''
R_\alpha$. Find $\alpha<\beta$, using Fact \ref{unbound}, such that 
\begin{equation*}
c\upharpoonright [\pi_{\oo_1}''R_\alpha, \pi_{\oo_1}'' R_\beta]> N.
   \end{equation*}
 Observe that $\tilde x_\alpha\nleq \tilde w_\beta$ by $\pi_\oo'' w_\beta= N<
c(\pi_{\oo_1}''\tilde x_\alpha,\pi_{\omega_1}''\tilde w_\beta)$ and (iii). Now find
$\gamma<\oo_1$ such that $(y_\gamma,w_\gamma)=(\tilde y_\alpha,\tilde w_\beta)$
and consider $t_\gamma\in \mb P_\gamma$. We claim that $t_\gamma$ is a minimal
extension of $\tilde y_\alpha$ and $\tilde w_\beta$ in the following sense:
\begin{enumerate}[(1)]
 \item $[\tilde x_\alpha,t_\gamma]=[\tilde x_\alpha,\tilde y_\alpha]\cup
\{t_\gamma\}$,
\item $[\tilde z_\beta,t_\gamma]=[\tilde z_\beta,\tilde w_\beta]\cup
\{t_\gamma\}$.
\end{enumerate}
Indeed, if $\tilde x_\alpha\leq t'<t_\gamma$ then $t'\leq \tilde y_\alpha$ or
$t'\leq \tilde w_\beta$; $\tilde x_\alpha\nleq \tilde w_\beta$ implies that
$t'\nleq w_\beta$ hence $t'\in [\tilde x_\alpha,\tilde y_\alpha]$. Similarly, if
$\tilde z_\beta\leq t'<t_\gamma$ then $t'\leq \tilde y_\alpha$ or $t'\leq \tilde
w_\beta$; however, $t'\nleq \tilde y_\alpha$ by $\pi_\omega'' t' > \pi_\omega''
\tilde y_\alpha$ so $t'\in[\tilde z_\beta,\tilde w_\beta]$.

Note that $t\in \mb P_0$ contradicts the 0-maximality of $[\tilde
x_\alpha,\tilde y_\alpha]$ and (1) while $t\in \mb P_1$ contradicts the
1-maximality  of $[\tilde z_\beta,\tilde w_\beta]$ and (2).
\end{proof}
The above claim finishes the proof.
\end{proof}

Using the previous theorem, we construct an irresolvable self-filling family; we
can actually realize this family as a system of open sets in a first countable
compact space. We remark that this space is base resolvable, as every compact
space, by Corollary \ref{Lindres}.

\begin{thm}\label{corson}
 There is a first countable Corson compact space $(X,\tau)$ and $\mc U\subseteq
\tau$ such that $\mc U$ fills $\{\cap \mc V: \mc V\in [\mc U]^{<\oo}\}$ and $\mc
U$ is irresolvable.
\end{thm}

\begin{proof} 
 Consider the poset $\mb P$ in Theorem \ref{part}. We say that $x\in [\mb
P]^\omega$ is a \emph{maximal chain} iff $\{x(n)\}_{n\in\oo}$ is a branch in $\mb
P$, $x(0)$ is a minimal element of $\mb P$ and $[x(n),x(n+1)]=\{x(n),x(n+1)\}$.
Note that there are no increasing chains of order type $\oo+1$ in $\mb P$.
Furthermore, since the intervals are finite   
\begin{obs}\label{maxchain}
 \begin{enumerate}
  \item Any branch $y\in [\mb P]^\oo$ can be extended to a maximal chain $\bar
y\in [\mb P]^\oo$,
\item there is an $n_0\in\oo$ such that $\cup_{n_0\leq n} [\bar y(n_0),\bar
y(n)]\subseteq\cup_{n\in\omega}[y(0),y(n)]$. 
 \end{enumerate}
\end{obs}
Note that (2) implies that if $y\in [\mb P]^\oo$ has homogeneous intervals with
respect to some coloring of $\mb P$ then the an end-segment of the maximal
extension $\bar y$ has the same property. 

Now consider $X=\{x\in [\mb P]^\oo: x \text{ is a maximal chain}\}$ as a subspace
of $2^{\mb P}$; here $2^{\mb P}$ is equipped with the usual product topology.

\begin{clm}
 $X$ is a compact subspace of $\Sigma(2^{\mb P})=\Sigma(2^{\oo_1})$.
\end{clm}
\begin{proof} $\Sigma(2^{\mb P})=\Sigma(2^{\oo_1})$ follows from $|\mb P|=\oo_1$
and clearly every chain is countable so $X\subseteq \Sigma(2^{\mb P})$.

We prove that $X$ is a closed subset of $2^{\mb P}$. Suppose that $y\in 2^{\mb
P}\setm X$; clearly, if $y$ is not a chain then $y$ can be separated from $X$.
Suppose that $y$ is a chain, then either $y(0)$ is not minimal in $\mb P$ or
there is $n\in\oo$ such that $[y(n),y(n+1)]\neq\{y(n),y(n+1)\}$. In the first case let
$\varepsilon\in Fn(\mb P,2)$ be defined to be 1 on $y(0)$ and $\varepsilon(p)=0$
for $p< y(0)$, $p\in \mb P$ (note that each element in $\mb P$ has only finitely
many predecessors); then $y\in [\varepsilon]$ and $[\varepsilon]\cap X=\empt$.
In the second case let $\varepsilon\in Fn(\mb P,2)$ such that
$1=\varepsilon(y(n))=\varepsilon(y(n+1))$ and $\varepsilon\uhr [y(n),y(n+1)]\setm\{y(n),y(n+1)\}=0$;
then $y\in [\varepsilon]$ and $[\varepsilon]\cap X=\empt$.
\end{proof} 

\begin{clm}
 $\{x\}=\cap\{[\chi_{x(n)}]\cap X:n\in\oo\}$ for every $x\in X$. Hence every
point in $X$ has countable pseudocharacter; in particular, $X$ is first
countable.
\end{clm}
\begin{proof}
Suppose that $y\in\cap\{[\chi_{x(n)}]\cap X:n\in\oo\}$, that is
$\{x(n):n\in\oo\}\subs \{y(n):n\in\oo\}$. We prove that $x(n)=y(n)$ by induction
on $n\in\oo$. $y(0)=x(0)$ as they are comparable minimal elements in $\mb P$. Suppose
that $x(i)=y(i)$ for $i< n$; if $x(n)\neq y(n)$ then $x(n)=y(k)$ for some $n<k$,
thus $y(n)\in [x(n-1),x(n)]=[y(n-1),y(k)]$ which contradicts the maximality $x$.
\end{proof}

Now define $$V_p=\{x\in X: \exists n\in\oo:x(n)\geq p\} \text{ for } p\in \mb
P,$$
and note that $V_p$ is open since $V_p=\cup\{[\chi_{\{q\}}]\cap X: p\leq q\}$.
We define $$\mc U=\{V_p:p\in \mb P\}.$$

\begin{clm}
 $\mc U$ fills $\{\cap \mc V: \mc V\in [\mc U]^{<\oo}\}$ and $\mc U$ is irresolvable.
\end{clm}

\begin{proof}
Note that $p< q$ in $\mb P$ if and only if $V_q\subsetneq V_p$; the nontrivial
direction is implied by property (2) of $\mb P$ in Theorem \ref{part}. 
To see that $\mc U$ fills the finite intersections from $\mc U$
let $\mc V\in [\mc U]^{<\oo}$ be arbitrary. If 
$A=\{p\in \mb P: V_p\in \mc V\}\in [\mb P]^{<\oo}$ then 
\begin{equation}\notag
\bigcap\mc V=\bigcup\{V_q: \text{ $p<q$ for all $p\in A$ }\}. 
\end{equation}

We show that $\mc U$ is irresolvable; suppose that we partitioned $\mc U$,
equivalently $\mb P$ into two parts $\mb P_0,\mb P_1$. Applying $\mb P \to
(I_\omega)^1_2$ we that there is a chain $y\in \mb P^\oo$ and $i<2$ such that
$[y(0),y(n)]\subseteq \mb P_i$ for every $n\in\omega$. By
Observation \ref{maxchain} there is maximal chain $\bar y\in X$ such that $[\bar y(n_0),\bar
y(n)]\subseteq \mb P_i$ for some $n_0\in\oo$ and every $n\geq n_0$. We claim
that there is no $V\in \{V_p:p\in \mb P_{1-i}\}$ such that $\bar y\in V\subseteq
V_{\bar y(n_0)}$. Indeed, if $\bar y\in V_p\subseteq V_{\bar y(n_0)}$ for some
$p\in \mb P$ then $\bar y(n_0)\leq p$ and there is $n\in\oo\setm n_0$ such that
$p\leq \bar y(n)$; that is $p\in [\bar y(n_0),\bar y(n)]\subseteq \mb P_i$. 
\end{proof}
The last claim finishes the proof of the theorem.
\end{proof}

Let us finish this section with the following:

\begin{lm}
 If $\mc U$ fills $\{\cap \mc V: \mc V\in [\mc
U]^{<\oo}\}$ and $\mc U$ is irresolvable then there is a non base resolvable, $T_0$ topological
space.
\end{lm}
\begin{proof} Suppose that $\mc U\subs \mc P(X)$ is as above. Define a relation $\sim$ on $X$ by $x\sim
y$ iff $\{U\in \mc U:x\in U\}=\{U\in \mc U:y\in U\}$; clearly, $\sim$ is an
equivalence relation on $X$. Let $[x]=\{x'\in X:x\sim x'\}$ for $x\in X$ and
let $[U]=\{[x]:x\in U\}$ for any $U\subs X$. It is clear that $[U]=\cup \{[V]:V\in \mc V\}$ if $U=\cup \mc V$ and $[U]=\cap \{[V]:V\in \mc V\}$ if $U=\cap \mc V$. Thus $\bs=\{[U]:U\in \mc U\}$ is a base for
a $T_0$ topology on $[X]$; sometimes this is referred to as the Kolmogorov quotient of the original (not necessarily $T_0$) topology generated by $\mc U$.

 It remains to show that $\bs$ is an irresolvable base. Take a partition $\bs=\bs_0\cup \bs_1$. Note that 
\begin{enumerate}[(1)]
\item $[x]\in [U]$ iff $x\in U$,
\item $[U]=[V]$ iff $U=V$,
\item $[U]\subs [V]$ iff $U\subs V$ 

\end{enumerate}
  for any $U,V\in \mc U$; thus the partition $\bs_0\cup \bs_1$ gives a partition $ \mc U_i=\{U\in \mc U: [U]\in \bs _i\}$ of $\mc U$. Now there is an $i<2$ so that $\mc U_i$ does not fill $\mc U$ i.e. there is $x\in X$ and $V\in \mc U$ so that $x\in U$ implies $U\setm V\neq \emptyset$ for all $U\in \mc U_i$. This gives that  $[x]\in [U]$ implies $[U]\setm [V]\neq \emptyset$ for all $[U]\in \bs _i$; in particular, $\bs _i$ is not a base for the topology generated by $\bs$.
\end{proof}

In particular, we have the following

\begin{cor}\label{T0irres}There is a non base resolvable, $T_0$ topological
space.
\end{cor}

\section{A 0-dimensional, Hausdorff space with an irresolvable base}
\label{conirres}

In this section, we partially strengthen Corollary \ref{T0irres} by showing

\begin{thm}
It is consistent that there is a first countable, 0-dimensional, $T_2$ space
which has a point countable,  irresolvable base. Furthermore, the space
has size $\mathfrak c$ and weight ${\omega}_1$. 
\end{thm}

\begin{proof}

For $\<{\alpha},n\>, \<{\beta},m\>\in {\omega}_1\times {\omega}$
write $\<{\alpha},n\>\triangleleft \<{\beta},m\>\in {\omega}_1\times {\omega}$
iff $\<{\alpha},n\>= \<{\beta},m\>$
or (${\alpha}<{\beta}$ and $n<m$ ).

\begin{dfn}\label{df:und_cup}
If $\preceq_1,\preceq_2\subs \triangleleft$, then let 
$\preceq_1\underline{\cup}\preceq_2$ be the partial order generated by
$\preceq_1{\cup}\preceq_2$.
 \end{dfn}

\begin{dfn}
If $\mc A=\<{\omega_1}\times {\omega}, \preceq\>$ is a poset with 
$\preceq\subs \triangleleft$, and for each ${\alpha}\in L_1$ 
we have a set $T_{\alpha}\subs {\alpha}\times {\omega}$ 
such that 
\begin{enumerate}[(C)]
 \item  \label{Cand}    
$\<T_{\alpha},\preceq\>$ is an everywhere ${\omega}$-branching tree, 
\end{enumerate}
then we say that the pair $\<\mc A, \<T_{\alpha}:{\alpha}\in L_1\>\>$
is a {\em candidate}.
\end{dfn}
Denote by $T_{\alpha}(n)$ the $n^{th}$ level of the tree 
$\<T_{\alpha},\preceq\>$.

\begin{dfn}
 Fix a candidate $\mbb A=\<\mc A, \<T_{\alpha}:{\alpha}\in L_1\>\>$.
We will define a topological space $X(\mbb A)$ as follows.

For ${\alpha}\in L_1$ let $\mc B(T_{\alpha})$ be the collection of the 
cofinal branches  of $T_{\alpha}$, and let
\begin{equation}\notag
\mc B(\mbb A)=\bigcup\{\mc B(T_{\alpha}):{\alpha}\in L_1\}. 
\end{equation}

The underlying set of the space $X(\mbb A)$ is $\mc B(\mbb A)$. 

For $x\in \ooto$ let $U(x)=\{y\in \ooto:x\preceq y\}$
and 
\begin{equation}\notag
 V(x)=\{b\in \mc B(\mbb A):\exists y\in b \ (x\preceq y)\}.
\end{equation}
Clearly $V(x)=\{b\in \mc B(\mbb A):b\subseteq^* U(x)\}$ where $\subseteq^*$ denotes containment modulo finite.

We declare that the  family
\begin{equation}\notag
\mc V=\{V(x):x\in \ooto\} 
\end{equation}
is the base of  $X(\mbb A)$.
\end{dfn}

\begin{lm}
$\mc V$ is a base and so 
$X(\mbb A)$ is a topological space. Moreover, $\mc V$ is point countable.  
\end{lm}
\begin{proof}
Assume that $b\in V(x)\cap V(y)$.
Then there is $z\in b$ such that $x\preceq z$ and $y\preceq z$.
Then $b\in V(z)\subs V(x)\cap V(y)$. 

To see that $\mc V$ is point countable, note that $b\notin V(x)$ if $b\in \mc B(T_\alpha)$ and $x\in (\omega_1\setm \alpha)\times \oo $. 
\end{proof}

For $x,y\in \ooto$ with $x\preceq y$ let
\begin{equation}\notag
 [x,y]=\{t\in \ooto: x\preceq t\preceq y\}.
\end{equation}

\begin{dfn}
We say that a candidate $\mbb A=\<\mc A, \<T_{\alpha}:{\alpha}\in L_1\>\>$
is {\em good} iff 
\begin{enumerate}[({G}1)]
\item \label{coherent} $V(u)\supset V(v)$ iff $u\preceq v$.
\item \label{Ta_cofinal}
$\forall {\alpha}\in L_1$ $\forall {\zeta}<{\alpha}$
$(T_{\alpha}\setm ({\zeta}\times {\omega}))\ne \empt$.
 \item\label{T2}
\begin{enumerate}[(a)]
 \item 
$\forall {\alpha}\in L_1$ $(\forall x,y\in T_{\alpha})$
$U(x)\cap U(y)\ne \empt$ iff $x$ and $y$ are $\preceq$-comparable.
\item
for each $\{{\alpha},{\beta}\}\in \br L_1;2;$ 
there is $f({\alpha},{\beta})\in {\omega}$ such that 
$$\text{$\forall x\in T_{\alpha}(f({\alpha},{\beta}))$ 
$\forall y\in T_{\beta}(f({\alpha},{\beta}))$
$U(x)\cap U(y)=\empt$}.$$
\end{enumerate}
\item\label{0-dim} For each $x\in \ooto$ and ${\alpha}\in L_1$
there is $g(x, {\alpha})\in{\omega}$ such that for each 
$y\in T_{\alpha}(g(x,{\alpha}))$
 $$\text{$U(y)\subs U(x)$ or $U(y)\cap U(x)=\empt.$}$$
\item\label{irres} If for  all ${\alpha}\in L_1$ 
and ${\zeta}<{\alpha}$
we choose a four element $\preceq$-increasing sequence  
\begin{equation}\notag
\<x^{\alpha}_{\zeta}, y^{\alpha}_{\zeta},z^{\alpha}_{\zeta},w^{\alpha}_{\zeta}
\>
\subs T_{\alpha}\setm ({\zeta}\times {\omega})
 \end{equation}
then there are $\{{\alpha},{\beta}\}\in \br L_1;2;$,
${\zeta}<{\alpha} $, ${\xi}<{\beta}$,
and $t\in T_{\alpha}\cap T_{\beta}$
such that  
\begin{enumerate}[(i)]
 \item $y^{\alpha}_{\zeta}\prec t$ and $[x^{\alpha}_{\zeta}, t]=
[x^{\alpha}_{\zeta}, y^{\alpha}_{\zeta}]\cup \{t\}$, 
\item $w^\beta_{\xi}\prec t$
and $[z^{\beta}_{\xi}, t]=
[z^{\beta}_{\xi}, w^{\beta}_{\xi}]\cup \{t\}$.
\end{enumerate}
\end{enumerate}

\end{dfn}

Basically (G3) will force the space to be Hausdorff, (G4) ensures that each $V(x)$ is clopen and (G5) will be used in proving irresolvability. Indeed, we have

\begin{lm}
If $\mbb A$ is a good candidate, then $X(\mbb A)$ is a  dense-in-itself, 
first countable,
 0-dimensional $T_2$ space  such that the base $\{V(x):x\in \ooto\}$ is
point countable and 
irresolvable.
\end{lm}

\begin{proof}
We prove this lemma in several steps.

 \begin{clm}
$X(\mbb A)$ is dense-in-itself.
\end{clm}

Indeed, assume that $b\in B(T_{\alpha})$ and $V(x)$ is an open 
neighbourhood  of $b$. Then there is $y\in b$ with $x\preceq y$
and so $b\in V(y)\subs V(x)$.
Thus $V(x)\supset V(y)\supset \{b'\in B(T_{\alpha}):y\in b'\} $,
and so $V(x)$  has $2^{\omega}$ many elements. So $b$ is not isolated. 

\begin{clm}
$X(\mbb A)$ is $T_2$. 
\end{clm}

Indeed, let $b\in B(T_{\alpha})$ and $c\in  B(T_{\beta})$ so that $b\neq c$.

If ${\alpha}={\beta}$ then pick $n\in \oo$ such that 
$x$,  the $n^{th}$ element of $b$, and 
$y$,  the $n^{th}$ element of $c$, are different.
Then $b\in V(x)$, $c\in V(y)$ and $V(x)\cap V(y)=\empt$ by (G\ref{T2})(a).

If ${\alpha}\ne {\beta}$ then write $n=f({\alpha},{\beta})$  (see G\ref{T2})(b)),
let  
$x$ be the   $n^{th}$ element of $b$, and let 
$y$ be   the $n^{th}$ element of $c$.
Then $b\in V(x)$, $c\in V(y)$ and $V(x)\cap V(y)=\empt$ by (G\ref{T2})(b).

\begin{clm}
Each set in $\{V(x):x\in \oo_1\times \oo\}$ is clopen, thus  $X(\mbb A)$ is 0-dimensional. 
\end{clm}

Indeed, assume that $x\in \ooto$, $b\in \mc B(T_{\alpha})$ and 
$b\notin V(x)$. Let $\{y\}= b\cap T_{\alpha}(g({\alpha},x))$.
Then $y\notin U(x)$ because $b\notin V(x)$, so $U(x)\cap U(y)=\empt$
by (G\ref{0-dim}). Thus $V(x)\cap V(y)=\empt$ as well.

\begin{clm}
The base $\{V(x):x\in \ooto\}$ is irresolvable.
 \end{clm}

Assume on the contrary that there is a partition $(K_0, K_1)$ of $\ooto$
such that both $\mc V_0=\{V(x):x\in K_0\}$ and 
$\mc V_1=\{V(x):x\in K_1\}$ are bases.

Assume that   ${\alpha}\in L_1$,  $x,y\in T_{\alpha}$ with $x\preceq y$ and  
$i\in 2$. We say that 
interval $[x,y]$ is  {\em $i$-maximal in $T_{\alpha}$} iff 
\begin{enumerate}[(i)]
 \item $[x,y]\subs K_i$, but $[x,z]\not\subs K_i$
for any $z$ with $y\prec z\in T_{\alpha}$.
\end{enumerate}

\begin{subclaim}
If ${\alpha}\in L_1$  and $x\in T_{\alpha}\cap K_i$,
then there is $x\preceq y\in T_{\alpha}$ such that the interval 
$[x,y]$ is $i$-maximal in $T_{\alpha}$.
\end{subclaim}

\begin{proof}[Proof of the Claim]
Assume on the contrary that there is no such $y$.
Then we can construct a strictly increasing sequence 
$\<x,y_0, y_1,\dots\>$ in $T_{\alpha}$ such that 
$[x, y_n]\subs K_i$ for all $n<{\omega}$.

Then $b=\{y\in T_{\alpha}:\exists n\in {\omega}\ y\preceq y_n\}\in \mc
B(T_{\alpha})$.

Since  $b\in V(x)$, and we assumed that
$\{V(z):z\in K_{1-i}\}$ is a base, 
 there is $z\in K_{1-i}$ with $b\in V(z)\subs V(x)$.  
Then $x\preceq z$ by (G\ref{coherent}). Moreover, there is $y\in b$
with $z\prec y$ because $b\in V(z)$. Thus $z\in [x,y]\cap K_{1-i}$, so 
$[x,y]\not\subs K_i$. Contradiction, the subclaim is proved.
\end{proof}

Using the subclaim,
for all $ {\alpha}\in L_1$ and for all ${\zeta}<{\alpha}$
we will construct  a four element $\preceq$-increasing sequence  
$$\<x^{\alpha}_{\zeta}, y^{\alpha}_{\zeta},z^{\alpha}_{\zeta},w^{\alpha}_{\zeta}
\>
\subs T_{\alpha}\setm ({\zeta}\times {\omega})$$
as follows. 

First, using (G\ref{Ta_cofinal}) pick 
$s^{\alpha}_{\zeta}\in T_{\alpha}\setm ({\zeta}\times {\omega})$.

If $K_0\cap U(s^{\alpha}_{\zeta})\cap T_{\alpha}=\empt$, then let 
$x^{\alpha}_{\zeta}=y^{\alpha}_{\zeta}=s^{\alpha}_{\zeta}$.

Otherwise pick  
$$x^{\alpha}_{\zeta}\in K_0\cap U(s^{\alpha}_{\zeta})\cap T_{\alpha},$$
and then, using the Subclaim above, pick 
\begin{equation}\notag
y^{\alpha}_{\zeta}\in U(x^{\alpha}_{\zeta})\cap T_{\alpha} 
\end{equation}
such that 
\begin{equation}\notag
\text{$[x^{\alpha}_{\zeta},y^{\alpha}_{\zeta}]$ is $0$-maximal in
$T_{\alpha}$}. 
\end{equation}

If $K_1\cap U(y^{\alpha}_{\zeta})\cap T_{\alpha}=\empt$, then let 
$z^{\alpha}_{\zeta}=w^{\alpha}_{\zeta}=y^{\alpha}_{\zeta}$.

Otherwise pick  
$$z^{\alpha}_{\zeta}\in K_1\cap U(y^{\alpha}_{\zeta})\cap T_{\alpha},$$
and then, using the Subclaim above, pick 
\begin{equation}\notag
w^{\alpha}_{\zeta}\in U(z^{\alpha}_{\zeta})\cap T_{\alpha} 
\end{equation}
such that 
\begin{equation}\notag
\text{$[z^{\alpha}_{\zeta},w^{\alpha}_{\zeta}]$ is $1$-maximal in
$T_{\alpha}$}. 
\end{equation}

By (G\ref{irres}),  there are $\{{\alpha},{\beta}\}\in \br L_1;2;$,
${\zeta}<{\alpha} $, ${\xi}<{\beta}$,
and $t\in T_{\alpha}\cap T_{\beta}$
such that  
\begin{enumerate}[(i)]
 \item $y^{\alpha}_{\zeta}\prec t$ and $[x^{\alpha}_{\zeta}, t]=
[x^{\alpha}_{\zeta}, y^{\alpha}_{\zeta}]\cup \{t\}$, 
\item $w^\beta_{\xi}\prec t$
and $[z^{\beta}_{\xi}, t]=
[z^{\beta}_{\xi}, w^{\beta}_{\xi}]\cup \{t\}$.
\end{enumerate}

Assume first that $t\in K_0$. Then 
$t\in K_0\cap  T_{\alpha}$, 
and $[x^{\alpha}_{\zeta}, t]=
[x^{\alpha}_{\zeta}, y^{\alpha}_{\zeta}]\cup \{t\}$, so 
$[x^{\alpha}_{\zeta}, t]\subs  K_0$, i.e. 
$[x^{\alpha}_{\zeta},y^{\alpha}_{\zeta}]$ was not $0$-maximal in $T_{\alpha}$.
Contradiction.
If $t\in K_1$, then a similar argument works 
using the interval $[z^{\beta}_{\xi},w^{\beta}_{\xi}]$ and $K_1$.

So in both cases we obtained a contradiction, so 
the base $\{V(x):x\in \ooto\}$ is irresolvable, 
which proves the lemma.
\end{proof}
 
Next we show that some c.c.c. forcing introduces a good candidate which finishes the proof the theorem.

Define the poset $\mc P=\<P,\le\>$ as follows.
The underlying set consists of  6-tuples
\begin{equation}\notag
 \<A, \preceq, I, \{T_{\alpha}:{\alpha}\in I\}, f,g \>,
\end{equation}
where
\begin{enumerate}[(P1)]
 \item\label{P-basic} $A\in \br \ooto;<{\omega};$, $\<A,\preceq\>$ is a poset,
$\preceq\subs \triangleleft$, $I\in \br \ooone;<{\omega};$, 
\item \label{P-tree}$T_{\alpha}\subs (A\cap {\alpha})\times {\omega}$ 
and $\<T_{\alpha},\preceq\>$ is a tree for ${\alpha}\in I$,
\item \label{P-fg}$f$ and $g$ are  functions, $\dom(f)\subs \br I;2;$, 
$\dom(g)\subs A\times I$, $\ran(f)\cup\ran(g)\subs {\omega}$
\item \label{P-T2}
To simplify our notation write $U(x)=\{y\in A: x\preceq x\}$ for $x\in A$.
\begin{enumerate}[(a)]
 \item 
 If ${\alpha}\in I$ and $x, y\in T_{\alpha}$ then 
$U(x)\cap U(y)\ne \empt $ iff $x$ and $y$ are $\preceq$-comparable.

\item If $\{{\alpha}, {\beta}\}\in \br \dom(f);2;$ and   $n=f({\alpha},{\beta})$,    
then 
\begin{equation}\notag
\text{
 $U[ T_{\alpha}(n)]\cap U[T_{\beta}(n)]=\empt$
 and $U[T_{\alpha}(n)]\cap T_{\beta}(<n)=\empt$.} 
\end{equation}
\end{enumerate}
\item \label{P-0-dim} if $\<x, {\alpha} \>\in \dom (g)$
then for all $y\in T_{\alpha}(g(x, {\alpha}))$
we have $U(y)\subs U(x)$ or $U(y)\cap U(x)=\empt$.  
\end{enumerate}

For $p\in P$
write $p=\<A^p, \preceq^p, I^p, \{T^p_{\alpha}:{\alpha}\in I^p\}, f^p,g^p \>$, 
and for $x\in A^p$ let $U^p(x)=\{y\in A^p:x\preceq^p y\}$.

For $p,q\in P$ let $p\le q$ iff
\begin{enumerate}[(O1)]
 \item $A^p\supset A^q$, and $\preceq^q=\preceq^p\restriction A_q$,
\item $I^p\supset I^q$ and $T^q_{\alpha}= T^p_{\alpha}\cap A^q$ for 
${\alpha}\in I^q$,
\item if $x\in A^p\setm A^q$, then $U^p(x)\cap A^q=\empt$,
\item $f^p\supset f^q$ and $g^p\supset g^q$,
\item if $U^q(x)\cap U^q(y)=\empt$ then  $U^p(x)\cap U^p(y)=\empt$.
\end{enumerate}

Clearly $\le$ is a partial order on $P$.

For $p\in P$ write $\supp(p)=I^p\cup \{{\alpha}:\<{\alpha},n\>\in A^p\text{
 for some }n\in {\omega}\} $.

If $\mc G$ is a $\mc P$-generic filter, then let
\begin{gather}\notag
A=\bigcup\{A^p:p\in \mc G\},\\\notag
\preceq=\bigcup\{\preceq^p:p\in \mc G\},\\\notag
I=\bigcup\{I^p:p\in \mc G\},\\\notag
T_{\alpha}=\bigcup\{T^p_{\alpha}:{\alpha}\in p\in \mc G\}
\text{ for ${\alpha}\in L_1$},\\\notag
f=\bigcup\{f^p:p\in \mc G\},\\\notag
g=\bigcup\{g^p:p\in \mc G\}.
\end{gather}

We  show that 
$\mc P$ satisfies c.c.c. and 
$\mbb A=\<\<\ooto,\preceq\>,\{T_{\alpha}:{\alpha}\in L_1\}\>$
is a good candidate.

\begin{dfn}
We say that the conditions $p$ and $q$ are {\em twins} iff conditions (T1)-(T7) below are satisfied:
\begin{enumerate}[(T1)]
 \item $|\supp(p)|=|\supp(q)|$, moreover   
$\max(\supp(p)\cap \supp(q))<\min(\supp(p)\bigtriangleup
\supp(q))$,
\end{enumerate}
Denote by $\rho$ the unique
order preserving  bijection between  $\supp(p)$ and $\supp(q)$, 
and define the function $\rhop:\supp(p)\times\omega  \to \supp(q)\times \omega$   
 by the formula
$\rhop(\<{\alpha},n\>)=\<\rho({\alpha}),n\>$.
\begin{enumerate}[(T1)]\addtocounter{enumi}{1}
\item $\rhop''A^p=A^q$ , 
\item $x \preceq^p y$  iff $\rhop(x)\preceq^q\rhop(y) $,
\item $\rho''I^p=I^q$,
\item $T^q_{\rho({\alpha})}=\rhop'' T_{\alpha}$,
\item $f^p(x,y)=m$ iff $f^q(\rhop(x),\rhop(y))=m$,
\item $g^p(x,{\alpha})=m$ iff $g^q(\rhop(x),\rho({\alpha}))=m$.
\end{enumerate}
\end{dfn}

\begin{lm}\label{lm:twin}
If $p$ and $q$ are twins
then 
\begin{multline}\notag
p\oplus q=\\\<A^p\cup A^q, \preceq^p\cup \preceq ^q, I^p\cup I^q,
 \{T^p_{\alpha}\cup T^q_{\alpha}: {\alpha}\in I^p\cup I^q\},f^p\cup f^q,
g^p\cup g^q\> 
\end{multline}
is a common extension of $p$ and $q$, where $T^p_{\alpha}=\empt$
for ${\alpha}\in I^q\setm I^p$ and $T^q_{\alpha}=\empt$
for ${\alpha}\in I^p\setm I^q$. 
\end{lm}

\begin{proof}
 Straightforward.
\end{proof}

\begin{lm}\label{lm:varphi}
There is a function $\varphi$ from  $P$ into some countable set such that if 
$\varphi(p)=\varphi(q)$ and $\supp(p)\cap \supp(q)<\supp(p)\bigtriangleup
\supp(q)$,
then $p$ and $q$ are twins. 
\end{lm}

\begin{proof}
Let $\varphi(p)$ be the type of the first order structure
\begin{equation}\notag
 \<\supp(p)\times {\omega},A^p,\preceq^p,  I^p,\{T^p_{\alpha}:{\alpha}\in I^p\},
f^p, g^p\>.
\end{equation}
\end{proof}

Lemmas \ref{lm:twin} and \ref{lm:varphi} yield that 
 $\mc P$ satisfies c.c.c

\begin{lm}
$A=\ooto$,  $I=L_1$  and $T_{\gamma}(0)\setm ({\zeta}\times {\omega})$ is 
infinite for all $\gamma\in L_1$ and ${\zeta}<\gamma$, and so  
(G\ref{Ta_cofinal}) holds.
\end{lm}

\begin{proof}
For  $p\in P$, $\gamma\in L_1$ and  $y\in (\gamma\times \omega )\setm A^p$ define 
$p\uplus \{y\}_{\gamma}$ as follows: 
\begin{multline}
\notag
p\uplus\{y\}_{\gamma}=\\
\<A^p\cup\{y\}, \preceq^p, I^p\cup \{\gamma\},
\{T^p_{\gamma}\cup\{y\}, T^p_{\alpha}:{\alpha}\in I^p\setm\{\gamma\}\}, 
f^p, g^p\>. 
 \end{multline}
Then $q=p\uplus\{y\}_{\gamma}\in P$ and $p\uplus\{y\}_{\gamma}\le p$.
If  $y\notin \zeta\times \omega$ then $q\Vdash y\in T_\gamma\setm (\zeta\times \omega)$ 
so  we are done.
\end{proof}

\begin{lm}\label{lm:p-x-y-gamma}
(a) Assume that $p\in P$,  
$a\in T^p_{\gamma}$ and 
 $b\in ({\gamma}\times {\omega})\setm A^p$ with  
$a\triangleleft b$. 
Let 
\begin{multline}\notag
p\uplus_a\{b\}_{\gamma}=\\\<A^p\cup\{b\}, \preceq^p\underline{\cup}\{\<a,b\>\}, 
\{T^p_{\gamma}\cup\{b\}, T^p_{\alpha}:{\alpha}\in I^p\setm\{\gamma\}\}, f^p,
g^p\>.  
\end{multline}
Then $p\uplus_a\{b\}_{\gamma}\in P$ and
$p\uplus_a\{b\}_{\gamma}\le p$. 

\medskip\noindent(b) 
The structure $\mbb A$ is a candidate. 
\end{lm}

\begin{proof}
First we check $q=p\uplus_a\{b\}_{\gamma}\in P$.

\smallskip\noindent
\pref{P-basic}-\pref{P-fg} are straightforward.
 
\smallskip\noindent
\pref{P-T2}(a):  Since $U^q(b)=\{b\}$, we can assume that $x,y\ne b$.
If $U^p(x)\cap U^p(y)\ne \empt$ then $x$ and $y$ are $\preceq^p$-comparable.
So we can assume that $b\in U^q(x)\cap U^q(y)$.
But then $a\in U^p(x)\cap U^p(y)$, so we are done.  

\smallskip\noindent
\pref{P-T2}(b):
Assume that $x\in T^q_{\alpha}(n)$, $y\in T^q_{\beta}(n)$ with $n=f^p(\alpha,\beta)=f^q(\alpha,\beta)$
and $z\in U^q(x)\cap U^q(y)$. If $z\ne b$ then
$z\in U^p(x)\cap U^p(y)$ which is not possible. So $z=b$.

If $x,y\ne b$, then  $a\in U^p(x)\cap U^p(y)$ which is not possible.
So we can assume that $x=b$ and $\alpha=\gamma$.
So $b\in T^q_\alpha(n)$ and so $a\in T^p_\alpha(n-1)$.
Thus $T^p_{\alpha}(n-1)\cap  U^p(y)\ne \empt$
which is not possible because \pref{P-T2}(b) holds for $p$.

Assume that $x\in T^q_{\alpha}(n)$, $y\in T^q_{\beta}(<n)$
and $y\in U^q(x)$. If $y\ne b$ then
$y\in U^p(x)\cap T^p_{\beta}(<n)$ which is not possible. So $y=b$
and $\beta=\gamma$.
Thus $a\in T^p_\beta(<n)\cap U^p_\alpha(x)$
which is not possible because \pref{P-T2}(b) holds for $p$.

\smallskip\noindent
\pref{P-0-dim}
Since $U(b)=\{b\}$, we can assume that $y\in A^p$.
Since   $b\in U^q(z)$ iff $a\in U^q(z)$ for $z\in A^p$,
if   $U^p(y)\subs U^p(x)$  then $U^q(y)\subs U^q(x)$,
and if   $U^p(y)\cap U^p(x)=\empt$  then 
$U^q(y)\cap U^q(x)=\empt$.

Thus we proved $q\in P$. Since $q\le p$ is straightforward, we are done.

\smallskip\noindent
(b) is clear from (a) by standard density arguments. 
\end{proof}

Now our aim is to prove that $\mb A$ is a good candidate.
\begin{lm}
$\mbb A$ has property (G\ref{coherent}). 
\end{lm}

\begin{proof}
Assume that $p\in P$, $u,v\in A^p$, $v\notin U^p(u)$. 
Pick $\gamma\in L_1\setm I^p$ with $\supp(p)\subs \gamma$, and
pick $b\in \gamma\times {\omega}$ with $v\triangleleft b$. 

Consider the condition {$q=p\uplus_v\{b\}_{\gamma}\le p$.}

Since $b\in T^q_{\gamma}$, we have $V(b)\cap \mc B(T_\gamma)\ne \empt$,
so $V(b)\ne \empt$.   Since $U^q(u)\cap U^q(b)=\empt$
we have $U(u)\cap U(b)=\empt$, and  so $V(u)\cap V(b)=\empt$,     
and so $\empt\ne V(b)\subs V(v)\setm V(u)$. 
\end{proof}

\begin{lm}
$\dom(f)=\br L_1;2;$ and $\dom(g)=(\ooto)\times L_1$.
Hence (G\ref{T2}) and (G\ref{0-dim})  holds.
\end{lm}

\begin{proof}
Assume that $\{{\gamma}, {\delta}\}\in \br I^p;2;\setm \dom(f^p)$.

Pick $m$ such that $T^p_{\alpha}(m)=\empt$ for all ${\alpha}\in I^p$.

Extends $f^p$ to $f^q$ as follows: 
$\dom(f^q)=\dom(f^p)\cup \big\{\{\gamma,\delta\}\big\}$
and $f^q(\gamma, \delta)=m$.

Let 
\begin{equation}\notag
q=\<A^p, \preceq^p, I^p, \{T^p_{\alpha}:{\alpha}\in I^p\}, f^q, g^p\>. 
\end{equation}
Then $q\in P$ and $q\le p$.

Similar argument works for $g$.
\end{proof}

Finally we  verify that (G\ref{irres}) also holds.

Assume that 
\begin{multline}\notag
 V^P\models \forall {\alpha}\in L_1\ \forall {\zeta}<{\alpha}\\
 \<x^{\alpha}_{\zeta}, y^{\alpha}_{\zeta},z^{\alpha}_{\zeta},w^{\alpha}_{\zeta}
\>
\subs T_{\alpha}\setm ({\zeta}\times {\omega})\text{ is $\preceq$-increasing.}
\end{multline}
For all $ {\alpha}\in L_1$ and $ {\zeta}<{\alpha}$ pick a condition 
$p^{\alpha}_{\zeta}=\<A^{\alpha}_{\zeta},\preceq^{\alpha}_{\zeta},\dots\>$ which
decides the sequence 
$\<x^{\alpha}_{\zeta}, y^{\alpha}_{\zeta},z^{\alpha}_{\zeta},w^{\alpha}_{\zeta}
\>$ and 
$\{x^{\alpha}_{\zeta}, y^{\alpha}_{\zeta},z^{\alpha}_{\zeta},w^{\alpha}_{\zeta}
\}\subs T^{\alpha}_{\zeta}$.

Let us  say that a $\Delta$-system $\mc A\subs \br \oo;<{\omega};$
is {\em nice} iff $A\cap B<A\bigtriangleup B$ for all $A\ne B\in \mc A$.

Using the Fodor lemma,
for each ${\zeta}\in \ooone$ find $m_{\zeta}<{\omega}$ and $I_{\zeta}\in \br
L_1;\ooone;$
such that 
\begin{enumerate}[(i)]
 \item $\varphi(p^{\alpha}_{\zeta})=m_{\zeta}$ for all $ {\alpha}\in I_{\zeta}$,
where $\varphi$ is from Lemma \ref{lm:varphi}.
 
\item $\{\supp(p^{\alpha}_{\zeta}):{\alpha}\in I_{\zeta}\}$
forms a nice $\Delta$-system with kernel $S_{\zeta}$,
moreover ${\alpha}\in \supp(p^{\alpha}_{\zeta})\setm S_{\zeta}$.
\item $\<x^{\alpha}_{\zeta},
y^{\alpha}_{\zeta},z^{\alpha}_{\zeta},w^{\alpha}_{\zeta}
\>=\<x_{\zeta}, y_{\zeta},z_{\zeta},w_{\zeta}
\>$ for ${\alpha}\in I_{\zeta}$.
\end{enumerate}
Then  $\{x^{\alpha}_{\zeta},
y^{\alpha}_{\zeta},z^{\alpha}_{\zeta},w^{\alpha}_{\zeta}
\}=\{x_{\zeta},
y_{\zeta},z_{\zeta},w_{\zeta}
\}\subs S_{\zeta}\times {\omega}$.

Find $m\in {\omega}$  and $I\in \br \ooone;\ooone;$ such that
\begin{enumerate}[(i)]\addtocounter{enumi}{3}
 \item $m_{\zeta}=m$ for all ${\zeta}\in I$, and so 
\begin{equation}\notag
 \forall {\zeta}\in I\ \forall {\alpha}\in I_{\zeta}\
\varphi(p^{\alpha}_{\zeta})=m.
\end{equation}
\item $\{S_{\zeta}:{\zeta}\in I\}$ forms a nice $\Delta$-system with kernel
$S$. 
\end{enumerate}
 
Pick $\{\xi,{\zeta}\}\in \br I;2;$.
Then pick ${\alpha}\in I_{\zeta}$ such that 
$S_{\xi}\cup S_{\zeta}<\supp(p^{\alpha}_{\zeta})\setm S_{\zeta}$.
So 
\begin{equation}\notag
S<(S_{\xi}\cup S_{\zeta})\setm S<\supp(p^{\alpha}_{\zeta})\setm S_{\zeta}.
 \end{equation}
Now pick ${\beta}\in I_\xi$ such that 
$\supp(p^{\alpha}_{\zeta})<\supp(p^{\beta}_\xi)\setm S_{\xi}$.
So 
\begin{equation}\notag
S<(S_{\xi}\cup S_{\zeta})\setm S<\supp(p^{\alpha}_{\zeta})\setm S_{\zeta}<
\supp(p^{\beta}_\xi)\setm S_{\xi}.
 \end{equation}
Thus $\supp (p^{\alpha}_{\zeta})\cap \supp (p^{\beta}_\xi)=S$,
${\alpha}\in \supp(p^{\alpha}_{\zeta})\setm S_{\zeta}$ and 
${\beta}\in \supp(p^{\beta}_\xi)\setm S_{\xi}$.

Since $\varphi (p^{\alpha}_{\zeta})= \varphi (p^{\beta}_\xi)$,
the conditions 
$\varphi (p^{\alpha}_{\zeta})$ and $ \varphi (p^{\beta}_\xi))$
are twins, and 
\begin{equation}\notag
q=p^{\alpha}_{\zeta}\oplus p^{\beta}_\xi 
\end{equation}
is a common extension.
Pick $t\in ({\alpha}\times {\omega})\setm (A^{\alpha}_{\zeta}\cup
A^{\beta}_{\zeta})$
with $y_{\zeta}\triangleleft t$ and $w_{\xi}\triangleleft t$.

Define $r$ as follows:
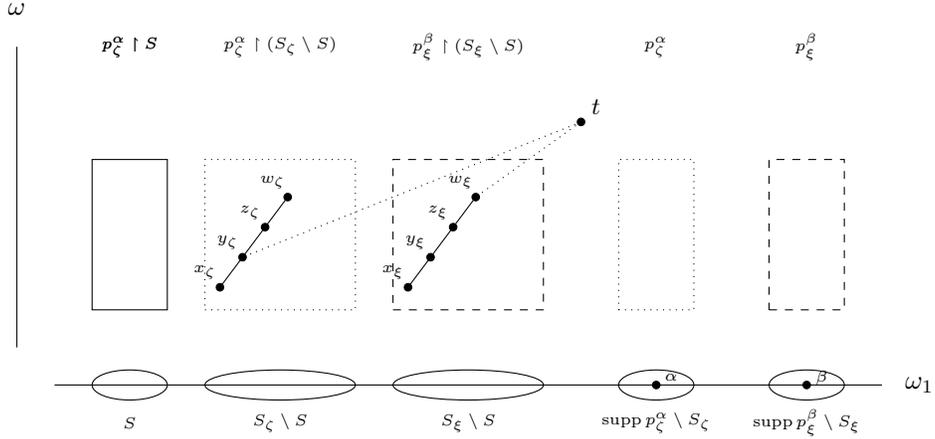
\begin{multline}\notag
r=\<A^q, \preceq_q\underline{\cup} \<y_{\zeta},t\>\underline{\cup} \<w_{\xi},t\>,I^q,\right .
\\\left.\{T^q_{\alpha}\cup\{t\},T^q_\beta\cup\{t\},T^\gamma:\gamma\in
I^q\setm\{{\alpha},{\beta}\}\},
f^q, g^q\>.
\end{multline}

\begin{figure}[h!]

\begin{tikzpicture}
 \path (0,0) coordinate (origin);
 \path (0,-0.5) coordinate (xstart);
 \path (11,-0.5) coordinate (xend);
 \path (-0.5,0) coordinate (ystart);
 \path (-0.5,4) coordinate (yend);
 \path (yend)  ++ (0,0.5)  coordinate (ylabel);
 \path (xend)  ++ (0.5,0)  coordinate (xlabel);

\path (1,-0.5) coordinate (S);
\path (1,-1) coordinate (Slabel);
\path (3,-0.5) coordinate (Szeta);
\path (3,-1) coordinate (Szetalabel);
\path (5.5,-0.5) coordinate (Sxi);
\path (5.5,-1) coordinate (Sxilabel);
\path (8,-0.5) coordinate (Szetaa);
\path (8,-1) coordinate (Szetaalabel);
\path (10,-0.5) coordinate (Sxib);
\path (10,-1) coordinate (Sxiblabel);
\path  (7,3) coordinate (t);
\path (t)  ++ (0.2,0.2) coordinate (tlabel);

\draw (xstart) -- (xend);
\draw (ystart) -- (yend);

\draw  (ylabel) node {$\omega$};
\draw  (xlabel) node {$\omega_1$};

\draw (S) ellipse (0.5 cm and 0.2cm);
\draw (Szeta) ellipse (1 cm and 0.2cm);
\draw (Sxi) ellipse (1 cm and 0.2cm);
\draw (Sxib) ellipse (0.5 cm and 0.2cm);
\draw (Szetaa) ellipse (0.5 cm and 0.2cm);

\draw (Slabel)  node {\tiny$S$};
\draw (Szetalabel)  node {\tiny$S_{\zeta}\setm S$};
\draw (Sxilabel)  node {\tiny$S_{\xi}\setm S$};
\draw (Szetaalabel)  node {\tiny $\supp p^{\alpha}_{\zeta}\setm S_{\zeta}$};
\draw (Sxiblabel)  node {\tiny $\supp p^{\beta}_{\xi}\setm S_{\xi}$};

\draw[fill] (t) circle (0.05cm); 
\draw (tlabel) node {\small $t$};

\draw (0.5, 0.5) rectangle (1.5, 2.5);  
\draw[dotted] (2, 0.5) rectangle (4, 2.5);  

\draw[dashed] (4.5, 0.5) rectangle (6.5, 2.5);

\draw[dotted] (7.5, 0.5) rectangle (8.5, 2.5);  

\draw[dashed] (9.5, 0.5) rectangle (10.5, 2.5);

\draw (1,4)  node {\tiny $p^{\alpha}_{\zeta}\restriction S$};

\draw (3,4)  node {\tiny $p^{\alpha}_{\zeta}\restriction (S_{\zeta}\setm S)$};

\draw (5.5,4)  node {\tiny $p^{\beta}_{\xi}\restriction (S_{\xi}\setm S)$};

\draw (1,4)  node {\tiny $p^{\alpha}_{\zeta}\restriction S$};

\draw (8,4)  node {\tiny $p^{\alpha}_{\zeta}$};

\draw (10,4)  node {\tiny $p^{\beta}_{\xi}$};

\draw[fill] (2.2,0.8) circle (0.05cm); 
\draw[fill] (2.5, 1.2) circle (0.05cm) coordinate (yz); 
\draw[fill] (2.8, 1.6) circle (0.05cm); 
\draw[fill] (3.1, 2.0 ) circle (0.05cm); 

\draw[fill] (2.2,0.8) ++ (-0.2,0.2)  node {\tiny $x_{\zeta}$}; 
\draw[fill] (2.5, 1.2) ++ (-0.2,0.2)  node {\tiny $y_{\zeta}$}; 
\draw[fill] (2.8, 1.6) ++ (-0.2,0.2)  node {\tiny $z_{\zeta}$}; 
\draw[fill] (3.1, 2.0 ) ++ (-0.2,0.2)  node {\tiny $w_{\zeta}$};

\draw (2.2,0.8) -- (2.5, 1.2)  -- (2.8, 1.6) -- (3.1, 2.0 ) ;

\draw[fill] (4.7,0.8) circle (0.05cm); 
\draw[fill] (5, 1.2) circle (0.05cm); 
\draw[fill] (5.3, 1.6) circle (0.05cm); 
\draw[fill] (5.6, 2.0 ) circle (0.05cm) coordinate (wx); 

\draw[fill] (4.7,0.8) ++ (-0.2,0.2)  node {\tiny $x_{\xi}$}; 
\draw[fill] (5, 1.2) ++ (-0.2,0.2)  node {\tiny $y_{\xi}$}; 
\draw[fill] (5.3, 1.6) ++ (-0.2,0.2)  node {\tiny $z_{\xi}$}; 
\draw[fill] (5.6, 2.0 ) ++ (-0.2,0.2)  node {\tiny $w_{\xi}$};

\draw (4.7,0.8)  -- (5, 1.2) -- (5.3, 1.6) -- (5.6, 2.0 );

\draw[dotted] (t) -- (yz);
\draw[dotted] (t) -- (wx);

\draw[fill] (Szetaa) circle (0.05cm); 
\draw[fill] (Sxib) circle (0.05cm); 

\draw (Szetaa) ++ (0.2,0.1) node {\tiny $ {\alpha}$};
\draw (Sxib) ++ (0.2,0.1) node {\tiny $ {\beta}$};

\end{tikzpicture}

\end{figure}

To check $r\in P$
we will use the following observation:
\begin{equation}\label{key1}
   r\restriction (\supp(p^{\alpha}_{\zeta})\cup \{t\})=
p^{\alpha}_{\zeta}\uplus_{y^{\alpha}_{\zeta}}\{t\}_{\alpha}
\end{equation}
and
\begin{equation}\label{key2}
   r\restriction (\supp(p^{\beta}_{\xi})\cup \{t\})=
p^{\beta}_{\xi}\uplus_{w^{\beta}_{\xi}}\{t\}_{\beta}.
 \end{equation}

Now let us check (P1)--(P5).

\smallskip\noindent{\pref{P-basic}} is trivial for $r$.

\smallskip\noindent{\pref{P-tree}}.  Let $\gamma\in I^q$.
If $\gamma\ne{\alpha},{\beta}$, then $T^q_\gamma=T^p_\gamma$,
so we are done.

Moreover, $T^r_{\alpha}=T^q_{\alpha}\cup\{t\}$,  $t\in {\alpha}\times
{\omega}$, 
and $\<T^r_{\alpha},\preceq\>$ is a tree by (\ref{key1}) and 
(\ref{key2}).

The same argument works for $T^r_{\beta}$.

\smallskip\noindent{\pref{P-fg}} is trivial.

\smallskip\noindent{\pref{P-T2}}(a).
Assume that $\gamma\in I^r$, $x,y\in T^r_{\gamma}$ with 
$U^r(x)\cap U^r(y)\ne \empt$. Since $U^r(t)=\{t\}$ we can assume
$x,y\in A^q$.  

Assume that $\gamma\in I^{\alpha}_{\zeta}$.
Then $T^q_\gamma\subs A^{\alpha}_{\zeta}$,
and so $x,y \in A^{\alpha}_{\zeta}$.
Thus  $t\in U^r(x)\cap U^r(y)$ implies $y^{\alpha}_{\zeta}\in U^r(x)\cap U^r(y)$.
So $U^q(x)\cap U^q(y)\ne \empt$, which yields that  $x$ and $y$
are $\preceq^q$ comparable because $q\in P$. 

Similar argument works when $\gamma\in I^{p^{\beta}_{\xi}}$.

\smallskip\noindent{\pref{P-T2}}(b).
Assume that $\{{\alpha}',{\beta}'\}\in \dom(f^r)=\dom(f^q)=
\dom(p^{\alpha}_{\zeta})\cup \dom(p^{\beta}_{\xi})$.   We
can assume that $\{{\alpha}',{\beta}'\}\in \dom(p^{\beta}_{\xi})$.

Write $n=f^r(\{{\alpha}', {\beta}'\})$.

(i) Assume on the contrary that there are $a\in T^r_{{\alpha}'}(n)$ and 
$b\in T^r_{{\beta}'}(n)$ with $U^r(a)\cap U^r(b)\ne \empt$.

First assume that $\{a,b\}\in \br A^q;2;$.
Since $q\in P$, we have $U^q(a)\cap U^q(b)=\empt$.
So $t\in U^r(a)\cap U^r(b)$ should hold.

If $c\in A^{\alpha}_{\zeta}$, then 
$t\in U(c)$ implies 
$y_{\zeta}\in U(c)$ by \ref{key1}.
Similarly, if $c\in A^{\beta}_{{\xi}}$, then 
$t\in U(c)$ implies 
$w_{\xi}\in U(c)$ by \ref{key2}.

Since $U^q(a)\cap U^q(b)=\empt$,
we can assume that $a\in A^{\alpha}_{\zeta}\setm
A^{\beta}_{\xi}$ 
and $b\in A^{\beta}_{\xi}\setm A^{\alpha}_{\zeta}$.

But then ${\alpha}'\in \supp({p^{\alpha}_{\zeta}})\setm S$
and ${\beta}'\in \supp({p^{\beta}_{\xi}})\setm S$,
so $f^r({\alpha}',{\beta}')$ is undefined. Contradiction.

So we can assume that e.g $t=a$ and $b\in A^q$.
Assume first that $b\in A^{p^{\alpha}_{\zeta}}$.
Then ${\alpha}'=\alpha$ and $y_{\zeta}\in A^{\alpha}_{\zeta}$
by (\ref{key1}).
Thus 
$y_{\zeta}\in T^{p^{\alpha}_{\zeta}}_{{\alpha}}(<n)\cap 
U^{p^{\alpha}_{\zeta}}(b)$, and so 
$T^{p^{\alpha}_{\zeta}}_{{{\alpha}}}(<n)\cap 
U[T^{p^{\alpha}_{\zeta}}_{{\beta}'}(n)]\ne \empt$, 
so \pref{P-T2}(b) fails for $p^{\alpha}_{\zeta}$.

If $b\in A^{\beta}_{\xi}$, then we can use similar arguments
using (\ref{key2}) instead of (\ref{key1}).

(ii)
Assume on the contrary that there are $a\in T^r_{{\alpha}'}(n)$ and 
$b\in T^r_{{\beta}'}(<n)\cap U^r(a) $.

Clearly $a\ne t$.
If $ b\ne t$, then $a\in T^q_{{\alpha}'}(n)$ and 
$b\in T^q_{{\beta}'}(<n)\cap U^q(a)$ which contradicts $q\in P$.

Assume that  $b=t$. 
If $b\in A^{p^{\alpha}_{\zeta}}$, then 
(\ref{key1}) implies ${\beta}'={\alpha}$ and 
$y_{\zeta}\in U^q(a)\cap T^q(<n)$. 
Thus $y_{\zeta}\in T^q_{{\beta}'}(<n)\cap U^q(a)$, which contradicts
$q\in P$.

If $b\in A^{p^{\beta}_{\xi}}$, then we can use similar arguments 
using (\ref{key2}) instead of (\ref{key1}).

\smallskip\noindent{\pref{P-0-dim}}.
Let $\<x, {\gamma} \>\in \dom (g^r)$ and  
$y\in T^r_{\gamma}(g(x, \gamma))$

Since $U^r(t)=\{t\}$, we can assume that $x,y\ne t$.

So $x,y\in A^q$. If $U^q(y)\subs U^q(x)$, then $x\preceq^q y$
and so $U^r(y)\subs U^r(x)$.

Assume on the contrary that  $U^q(x)\cap U^q(y)=\empt$, but $t\in U^r(x)\cap
U^r(y)$. 

We can assume that $\<x,\gamma\>\in g^{p^{\alpha}_{\zeta}}$.
Thus $x\in A^{\alpha}_{{\zeta}}$ and $\gamma\in I^{\alpha}_{{\zeta}}$.

However $T^q_\gamma\subs A^{\alpha}_{\zeta}$, so $y\in A^{\alpha}_{\zeta}$.

Since $x,y\in A^{\alpha}_{\zeta}$ and $\gamma \in I^{\alpha}_{\zeta}$, 
$t\in  U^r(x)\cap U^r(y)$ implies 
$y_{\zeta}\in  U^{p^{\alpha}_{\zeta}}(x)\cap U^{p^{\alpha}_{\zeta}}(y)$ 
by (\ref{key1}), which contradicts  $U^q(x)\cap U^q(y)=\empt$.

\medskip

So we proved 
$ r\in P$.

\medskip

Next we show  that $r\le p^{\alpha}_{\zeta}, p^{\beta}_{\xi}$.
(O1)--(O4) are trivial. To check (O5), assume on the contrary that 
$U^{p^{\alpha}_{\zeta}}(a)\cap U^{p^{\alpha}_{\zeta}}(b)=\empt$,
but  $U^r\cap U^r(b)\ne\empt$.

Then $t\in U^r(a)\cap U^r(b)$,
and so $y^{\alpha}_{\zeta}\in U^{p^{\alpha}_{\zeta}}(a)\cap
U^{p^{\alpha}_{\zeta}}(b)$ by (\ref{key1}),
which is a contradiction.

\medskip

Finally, it is also straightforward that 
\begin{equation}
r\Vdash \text{ (G\ref{irres})(i)--(ii)  holds for
${\alpha},{\beta},{\zeta},{\xi}$, and $t$.} 
\end{equation}
So we proved the theorem.
\end{proof}

\section{Open problems} \label{problems}

In this section, we present a list of  open problems which could be of further
interest and are closely connected to our results.

\begin{prob}Is every linearly ordered space base resolvable? 	
\end{prob}

\begin{prob} Is every $T_3$ (hereditarily) separable space base resolvable?
\end{prob}

\begin{prob} Is every paracompact space base resolvable?
\end{prob}

Note that under PFA, every $T_3$ hereditarily separable space is Lindel\"of
hence base resolvable by Corollary \ref{Lindres}. Also, we conjecture that our
forcing construction can be modified to produce a separable non base resolvable
space.

\begin{prob} Is every power of $\mb R$ base resolvable? Is it true that base
resolvability is preserved by products?
\end{prob}

We know that every $\pi$-base is the union of two disjoint $\pi$-bases by
Proposition \ref{first}(2). However:

\begin{prob} Does every base contain a disjoint base and $\pi$-base? 
\end{prob}

Bases closed under finite unions are resolvable by Corollary \ref{finunion} which
raises to following question:

\begin{prob} Is it true that every base which is closed under finite intersections
is base resolvable?
\end{prob}

It would be interesting to look into the following:

\begin{prob} Is every self filling family $\mc F$ of closed (Borel) sets of $\oo^\oo$ resolvable? 
\end{prob}

Concerning negligible subsets we ask the following:

\begin{prob} Is there a base $\bs$ for some space $X$ such that every $\mc U \in
[\bs]^{|\bs|}$ contains a neighborhood base at some point?
\end{prob}

\end{document}